\documentclass[11pt]{amsart}

\usepackage{amssymb}
\usepackage{mathrsfs}
\usepackage{amsfonts}
\usepackage{latexsym,amsmath,amsthm,amssymb,amsfonts}
\usepackage[usenames]{color}
\usepackage{xspace,colortbl}
\usepackage{graphicx}
\usepackage{tipa}

\newcommand{\be}{\begin{equation}}
\newcommand{\ee}{\end{equation}}
\newcommand{\beq}{\begin{eqnarray}}
\newcommand{\eeq}{\end{eqnarray}}

\newtheorem{prop}{Proposition}[section]

\newtheorem{remark}[prop]{Remark}

\def\begeq{\begin{equation}}
\def\endeq{\end{equation}}

\def\odot{\setbox0=\hbox{$\bigcirc$}\relax \mathbin {\hbox
to0pt{\raise.5pt\hbox to\wd0{\hfil $\wedge$\hfil}\hss}\box0 }}

\numberwithin{equation} {section}

\numberwithin{equation}{section}
\newtheorem{theorem}{\bf Theorem}[section]
\newtheorem{proposition}[theorem]{\bf Proposition}
\newtheorem{definition}[theorem]{\bf Definition}
\newtheorem{lemma}[theorem]{\bf Lemma}

\allowdisplaybreaks

\begin{document}
\title[The Dirichlet problem for a class of Hessian quotient equations ]
 {The Dirichlet problem for a class of Hessian quotient \\ equations in Lorentz-Minkowski space $\mathbb{R}^{n+1}_{1}$}

\author{
 Ya Gao,~~YanLing Gao,~~Jing Mao$^{\ast}$}

\address{
Faculty of Mathematics and Statistics, Key Laboratory of Applied
Mathematics of Hubei Province, Hubei University, Wuhan 430062, China
}

\email{Echo-gaoya@outlook.com, 1766019437@qq.com, jiner120@163.com}

\thanks{$\ast$ Corresponding author}

\date{}
\maketitle
\begin{abstract}
In this paper, under suitable settings, we can obtain the existence
and uniqueness of solutions to a class of Hessian quotient equations
with Dirichlet boundary condition in Lorentz-Minkowski space
$\mathbb{R}^{n+1}_{1}$, which can be seen as a prescribed curvature
problem and a continuous work of \cite{GLM}.
\end{abstract}

\maketitle {\it \small{{\bf Keywords}: Spacelike hypersurfaces,
Lorentz-Minkowski space, Hessian quotient equations, curvature
estimates, Dirichlet boundary condition.

 }

}

\section{Introduction}
Throughout this paper, let $\mathbb{R}^{n+1}_{1}$ be the
$(n+1)$-dimensional ($n\geq2$) Lorentz-Minkowski space with the
following Lorentzian metric
\begin{eqnarray*}
\langle\cdot,\cdot\rangle_{L}=dx_{1}^{2}+dx_{2}^{2}+\cdots+dx_{n}^{2}-dx_{n+1}^{2}.
\end{eqnarray*}
In fact, $\mathbb{R}^{n+1}_{1}$ is an $(n+1)$-dimensional Lorentz
manifold with index $1$. Denote by
\begin{eqnarray*}
\mathscr{H}^{n}(1)=\{(x_{1},x_{2},\cdots,x_{n+1})\in\mathbb{R}^{n+1}_{1}|x_{1}^{2}+x_{2}^{2}+\cdots+x_{n}^{2}-x_{n+1}^{2}=-1~\mathrm{and}~x_{n+1}>0\},
\end{eqnarray*}
which is exactly the hyperbolic plane of center $(0,0,\ldots,0)$
(i.e., the origin of $\mathbb{R}^{n+1}$) and radius $1$ in
$\mathbb{R}^{n+1}_{1}$.

Assume that
 \begin{eqnarray} \label{G1}
\mathcal{G}:=\{(x,u(x))|x\in M^{n} \subset \mathscr{H}^{n}(1)\}
\end{eqnarray}
is a spacelike graphic hypersurfaces defined over some bounded piece
$M^{n} \subset \mathscr{H}^{n}(1)$, with the boundary $\partial
M^{n}$, of the hyperbolic plane $\mathscr{H}^{n}(1)$, where
$\sup_{M^n}\frac{|Du|}{u}\leq\rho<1$. Let $x$ be a point on
$\mathscr{H}^{n}(1)$ which is described by local coordinates
$\xi^{1},\ldots,\xi^{n}$, that is, $x=x(\xi^{1},\ldots,\xi^{n})$. By
the abuse of notations, let $\partial_i$ be the corresponding
coordinate vector fields on $\mathscr{H}^{n}(1)$ and
$\sigma_{ij}=g_{\mathscr{H}^{n}(1)}(\partial_i,\partial_j)$ be the
induced Riemannian metric on $\mathscr{H}^{n}(1)$. Of course,
$\{\sigma_{ij}\}_{i,j=1,2,\ldots,n}$ is also the metric on
$M^{n}\subset\mathscr{H}^{n}(1)$. Denote by\footnote{~Clearly, for
accuracy, here $D_{i}u$ should be $D_{\partial_{i}}u$. In the
sequel, without confusion and if needed, we prefer to simplify
covariant derivatives like this. In this setting,
$u_{ij}:=D_{j}D_{i}u$, $u_{ijk}:=D_{k}D_{j}D_{i}u$ mean
$u_{ij}=D_{\partial_{j}}D_{\partial_{i}}u$ and
$u_{ijk}=D_{\partial_{k}}D_{\partial_{j}}D_{\partial_{i}}u$,
respectively. } $u_{i}:=D_{i}u$, $u_{ij}:=D_{j}D_{i}u$, and
$u_{ijk}:=D_{k}D_{j}D_{i}u$ the covariant derivatives of $u$ w.r.t.
the metric $g_{\mathscr{H}^{n}(1)}$, where $D$ is the covariant
connection on $\mathscr{H}^{n}(1)$. Let $\nabla$ be the Levi-Civita
connection of $\mathcal{G}$ w.r.t. the metric
$g:=u^{2}g_{\mathscr{H}^{n}(1)}-dr^{2}$ induced from the Lorentzian
metric $\langle\cdot,\cdot\rangle_{L}$ of $\mathbb{R}^{n+1}_{1}$.
Clearly, the tangent vectors of $\mathcal{G}$ are given by
\begin{eqnarray*}
X_{i}=(1,Du)=\partial_{i}+u_{i}\partial_{r}, \qquad i=1,2,\ldots,n.
\end{eqnarray*}
The induced metric $g$ on $\mathcal{G}$ has the form
\begin{equation*}\label{g_{ij}}
g_{ij}=\langle X_{i},X_{j}\rangle_{L}=u^2\sigma_{ij}-u_{i}u_{j},
\end{equation*}
and its inverse is given by
\begin{equation*}\label{g^{ij}}
g^{ij}=\frac{1}{u^2}\left(\sigma^{ij}+\frac{u^i  u^j
}{u^2v^{2}}\right).
\end{equation*}
Then the future-directed timelike unit normal of $\mathcal{G}$ is
given by
\begin{eqnarray*}\label{nu}
\nu=\frac{1}{v}\left(\partial_r+\frac{1}{u^2}u^j\partial_j\right),
\end{eqnarray*}
where $u^{j}:=\sigma^{ij}u_{i}$ and $v:=\sqrt{1-u^{-2}|D u|^2}$ with
$D u$ the gradient of $u$. Of course, in this paper we use the
Einstein summation convention -- repeated superscripts and
subscripts should be made summation from $1$ to $n$. The second
fundamental form of $\mathcal{G}$ is
\begin{equation}\label{h_{ij}}
h_{ij}=-\langle\overline{\nabla}_{X_{j}}X_{i},\nu\rangle_{L}
=\frac{1}{v}\left(u_{ij}+u \sigma_{ij}-\frac{2}{u}{u_i u_j }\right),
\end{equation}
 with $\overline{\nabla}$ the covariant connection in
 $\mathbb{R}^{n+1}_{1}$.
Denote by $\lambda_{1},\lambda_{2},\ldots,\lambda_{n}$ the principal
curvatures of $\mathcal{G}$, which are actually the eigenvalues of
the matrix $(h_{ij})_{n\times n}$ w.r.t. the metric $g$. The
so-called \emph{$k$-th Weingarten curvature} at
$X=(x,u(x))\in\mathcal{G}$ is defined as
 \begin{eqnarray}  \label{kwc}
\sigma_{k}(\lambda_{1}, \lambda_{2}, \cdots,
\lambda_{n})=\sum\limits_{1\leq i_{1}<i_{2}<\cdots<i_{k}\leq
n}\lambda_{i_{1}}\lambda_{i_{2}}\cdots\lambda_{i_{k}}.
 \end{eqnarray}

We also need the following conception:
\begin{definition}
For $1\leq k\leq n$, let $\Gamma_{k}$ be a cone in $ \mathbb{R}^{n}$
determined by
\begin{eqnarray*}
\Gamma_{k}=\{\lambda\in\mathbb{R}^{n}|\sigma_{j}(\lambda)>0,
~~j=1,2,\ldots,k\}.
\end{eqnarray*}
A smooth spacelike graphic hypersurface
$\mathcal{G}\subset\mathbb{R}^{n+1}_{1}$ is called $k$-admissible if
at every point $X\in\mathcal{G}$,
 $(\lambda_{1},\lambda_{2},\ldots,\lambda_{n})\in\Gamma_{k}$.
\end{definition}

\begin{remark}
\rm{ We kindly refer readers to \cite[Remark 1.1]{GLM} for a brief
introduction to the $\sigma_{k}$ operator on
$\mathcal{G}\subset\mathbb{R}^{n+1}_{1}$ and our previous works
(see, e.g., \cite{GaoY2,gm2,gm21,gm3,glm}) related to this operator.
}
\end{remark}

In this paper, we investigate the existence and uniqueness of
solutions for a class of nonlinear partial differential equations
(PDEs for short) given as follows
\begin{equation}\label{main equation}
\left\{
\begin{aligned}
& \frac{\sigma_{k}}{\sigma_{l}}=\psi(x,u,\vartheta), \qquad &&x\in
M^{n}\subset \mathscr{H}^{n}(1)\subset\mathbb{R}^{n+1}_{1},\quad
2\leq k\leq n,\quad 0\leq l\leq k-2,
\\
&u=\varphi, \qquad&&x\in \partial M^{n},
\end{aligned}
\right.
\end{equation}
where $\psi$, depending on $X:=(x,u)$, $\vartheta:=-\langle X,
\nu\rangle_{L}$, and $\varphi$ are functions defined on $M^{n}$.

\begin{remark}
\rm{ (1) Obviously, when $l=0$, the LHS of the first equation in
(\ref{main equation}) degenerates into $\sigma_{k}(\lambda_{1},
\lambda_{2}, \cdots, \lambda_{n})$ exactly -- the $k$-th Weingarten
curvature of $\mathcal{G}$, and then the system (\ref{main
equation}) becomes the prescribed curvature problem (PCP for short)
studied in \cite{GLM} with\footnote{~The case $k=1$ has also been
discussed in \cite{GLM}.} $k=2,3,\cdots,n$. In this sense, we also
say that (\ref{main equation}) is a PCP and should be a continuity
of our previous work \cite{GLM}. \\
 (2) In fact, our experience in  \cite{GLM} tells that if $M^{n}$ in the PCP (\ref{main equation}) was
 replaced by
 bounded domain
 $\Omega\subset\mathbb{R}^{n}\subset\mathbb{R}^{n+1}_{1}$, then similar
 existence and uniqueness conclusion (as in Theorem \ref{main 1.3} below)
 for the new PCP
 could be expected provided suitable assumptions made to the positive function $\psi$ and the Dirichlet boundary condition (DBC for short)
 $\varphi$. We prefer to leave this problem as an exercise for
 readers since the possible a priori estimates can be carried out by
 using the ones (we have shown here) for reference. \\
 (3) The PCPs (with or without boundary condition) in Euclidean space
or even more general Riemannian manifolds were extensively studied
-- see, e.g., \cite{cns1,cns2,gm,WJJ} and the references therein for
details. Affected by the study of Geometry of Submanifolds, it is
natural to consider PCPs in the pseudo-Riemannian context. In fact,
 many other
important results on PCPs in pseudo-Riemannian manifolds have been
obtained. For instance, in the Lorentz-Minkowski space or general
Lorentz manifolds, Bartnik \cite{b1}, Bartnik-Simon \cite{bs},
Gerhardt \cite{cg1,cg2} solved the Dirichlet problem for the
prescribed mean curvature equation, Delano\`{e} \cite{fd}, Guan
\cite{gb} considered the prescribed Gauss-Kronecker curvature
equation with DBC, while Bayard \cite{Bayard}, Gerhardt \cite{cg3},
Urbas \cite{ju2} worked for the prescribed scalar curvature
equation. From this brief introduction, one can see that when
considering the PCPs (with or without boundary condition) in either
Riemannian manifolds or pseudo-Riemannian manifolds, equations
involving the $\sigma_{k}$ operator of the second fundamental form
$A$ were considered a lot -- this is natural, since in
general\footnote{~ Clearly, $\sigma_{k}(\lambda(\cdot))$ denotes the
$k$-th elementary symmetric function of eigenvalues of a given
tensor -- the second fundamental form $A$. }
$\sigma_{k}(\lambda(A))$ corresponds to the $k$-th Weingarten
curvature of submanifolds in some geometric space considered. Our
research experience in curvature flows inspires us that maybe it is
possible to improve our previous existence and uniqueness result
\cite[Theorem 1.4]{GLM} to more general setting -- replacing
$\sigma_{k}(\lambda(A))$ by the $(k,l)$-Hessian quotient
$\frac{\sigma_{k}(\lambda(A))}{\sigma_{l}(\lambda(A))}$ of
$\lambda(A)$, with $2\leq k\leq n, 0\leq l\leq k-2$. However, in the
Riemannian context, Guan-Ren-Wang \cite{grw} showed that
$C^2$-estimates fail for the curvature equation of the form
\begin{eqnarray*}
\frac{\sigma_{k}(\lambda(A))}{\sigma_{l}(\lambda(A)}=f(X,\nu(X)),
\qquad \forall X\in N^{n}\subset\mathbb{R}^{n+1},
\end{eqnarray*}
 where $\nu(X)$ therein denotes the outward unit normal vector of
 closed
 hypersurface $N^{n}$ (in the Euclidean
 $(n+1)$-space $\mathbb{R}^{n+1}$) at $X$. This fact brings us
 negative attitude to consider the PCP (\ref{main equation}), and
 one might worries about that maybe same story happens in the
 pesudo-Riemannian context. But luckily, we \emph{almost} overcome this difficulty
 and successfully obtain the existence and uniqueness of solutions
 to the PCP (\ref{main equation}) for some special $k$ and $l$ -- see Theorem \ref{main 1.3}
 below for details.
}
\end{remark}

For the PCP \eqref{main equation}, first, we can get the following
curvature maximum estimate:

\begin{theorem}\label{main1.1}
 Suppose that $u\in C^{4}(M^{n})\cap C^{2}(\overline{M^{n}} )$ is a spacelike, $k$-admissible solution of the PCP
 \eqref{main equation}, $0<\psi\in C^{\infty}(\overline{M^{n}}) $ and that $\psi^{\frac{1}{k-l}}(X,\vartheta)$ is convex in $\vartheta$ and satisfies
\begin{equation}\label{sf's conditions}
\frac{\partial \psi^{\frac{1}{k-l}}(X,\vartheta)}{\partial
\vartheta}\cdot \vartheta\geq\psi^{\frac{1}{k-l}}(X,\vartheta) \qquad
for~ fixed~ X\in\mathcal{G}.
\end{equation}
Then the second fundamental form $A$ of $\mathcal{G}$ satisfies
\begin{equation}\label{A's boundary}
\sup\limits_{M^{n}}||A||\leq C\left(1+\sup\limits_{\partial
M^{n}}||A||\right),
\end{equation}
where $C$ depends only on $n$, $||\varphi||_{C^{1}(\overline{M^{n}}
)}$, $||\psi||_{C^{2}\left(\overline{M^{n}}\times\left[\inf\limits_{\partial M^{n}}u,\sup\limits_{\partial M^{n}}u\right]\times\mathbb{R}\right)}$.\\
\end{theorem}

\begin{remark}
\rm{ It is not hard to find some $\psi$ satisfying assumptions in
Theorem \ref{main1.1}. For instance, (i)
$\psi(x,u,\vartheta)=\vartheta^{p}h(x,u)$ for $p\geq k-l$; (ii)
$\psi(x,u,\vartheta)=e^{p\frac{\vartheta}{u}}h(x,u)$ for $p\geq
k-l$. }
\end{remark}

The following interior curvature estimate can also be obtained:

\begin{theorem}\label{main1.2}
Suppose that $u\in C^{4}(M^{n})\cap C^{2}(\overline{M^{n}} )$ is a
spacelike, $k$-admissible solution of the PCP \eqref{main equation},
$0<\psi\in C^{\infty}(\overline{M^{n}}) $ and that
$\psi^{\frac{1}{k-l}}(X,\vartheta)$ is convex in $\vartheta$ and
satisfies
\begin{equation}\label{f's conditions}
\frac{\partial \psi^{\frac{1}{k-l}}(X,\vartheta)}{\partial
\vartheta}\cdot \vartheta>\psi^{\frac{1}{k-l}}(X,\vartheta) \qquad
for~ fixed~ X\in\mathcal{G}.
\end{equation}
Furthermore, suppose that $M^{n}\subset\mathscr{H}^{n}(1)$ is
$C^{2}$ and uniformly convex, and that $\varphi$ is spacelike and
affine. If $u\in C^{4}(M^{n})$ is a spacelike, $k$-admissible
solution of the PCP \eqref{main equation}, then
\begin{equation*}\label{f's condition}
\sup\limits_{\widetilde{M^{n}}}|A|\leq C\left(\widetilde{M^{n}}\right)
\end{equation*}
for any $\widetilde{M^{n}}\subset\subset M^{n}$, where $C\left(\widetilde{M^{n}}\right)$ depends only on $n$, $\zeta$, $M^{n}$, $\mathrm{dist}(\widetilde{M^{n}},\partial M^{n})$, $||\varphi||_{C^{1}(\overline{M^{n}})}$, and $||\psi||_{C^{2}\left(\overline{M^{n}}\times\left[\inf\limits_{\partial M^{n}}u,\sup\limits_{\partial M^{n}}u\right]\times\mathbb{R}\right)}$.
\end{theorem}

Combining the above curvature estimates and the boundary
$C^{2}$-estimates proven in the sequel, we can get the following
result:

\begin{theorem}\label{main 1.3}
For $k=2$, $l=0$, suppose that $M^{n}$ is a smooth bounded domain of
$\mathscr{H}^{n}(1)$ and is strictly convex, while $\psi$ is a
smooth positive function and $\psi^{\frac{1}{2}}$ is convex in
$\vartheta$ satisfying
\begin{equation*}\label{f's condition}
\frac{\partial \psi^{\frac{1}{2}}(X,\vartheta)}{\partial
\vartheta}\cdot \vartheta\geq\psi^{\frac{1}{2}}(X,\vartheta) \qquad
for~ fixed~ X\in\mathcal{G}.
\end{equation*}
Then for any spacelike, affine function $\varphi$, there exist a uniquely smooth spacelike, $2$-admissible graphic hypersurface $\mathcal{G}$ (define over $M^{n}$) with
 the prescribed $(2,0)$-Hessian quotient
$\frac{\sigma_{2}(\lambda(A))}{\sigma_{0}(\lambda(A))}=\sigma_{2}(\lambda(A))=\psi$ and Dirichlet boundary data $\varphi$.\\
\end{theorem}

\begin{remark}
\rm{ (1) Although nearly the whole part of
  the a priori estimates works for $2\leq k\leq n,0\leq l\leq k-2$, the estimate for the double
  normal second derivatives
on the boundary only works for $k=2$, which leads to the situation
that so far, we can get the existence and uniqueness of solutions to
the PCP \eqref{main equation} only for $k=2$, $l=0$. However, we do
hope that this restriction can be overcome in the future, i.e., the
estimate for the double
  normal second derivatives
on the boundary can also be obtained for $3\leq k\leq n$. If so,
that would be a surprising breakthrough. \\
 (2) Clearly, when $k=2$, $l=0$, the existence and uniqueness conclusion of Theorem \ref{main
 1.3} is the same with our previous result \cite[Theorem 1.4]{GLM}.
 But, the reason why we still write down this paper is that the a
 priori estimates here (especially the part of the $C^2$ boundary estimates) is much complicated than the one in
 \cite{GLM}. In fact, because of this reason, we did not show the
 details of the $C^2$ boundary estimates for the solutions to the
 PCP considered in \cite{GLM}. \\
  (3) In our previous work \cite{GLM} and this paper, we insist on
  numbering (by subscripts) nearly all constants appearing in the process of doing a
  priori estimates, and we believe that this way can
  reveal the relations among constants clearly to readers.
}
\end{remark}

The paper is organized as follows. Some useful formulae  for
spacelike graphic hypersurfaces defined over
$M^{n}\subset\mathscr{H}^{n}(1)$ will be introduced in Section
\ref{S2}. These formulae have been proven carefully in \cite{GLM}.
Section \ref{c1es} devotes to the gradient estimate. Curvature
maximum principle will be shown in Section \ref{S4}. Interior
$C^{2}$-estimates will be proven in Section \ref{S5}, and boundary
$C^{2}$-estimates will be proven in Section 6, which, together with
the method of continuity, lead to the existence and uniqueness
result for the PCP \eqref{main equation}, i.e. Theorem \ref{main
1.3}.

\section{Some elementary formulas} \label{S2}
For the spacelike graphic hypersurface
$\mathcal{G}\subset\mathbb{R}^{n+1}_{1}$ given by \eqref{G1} and
$X=(x,u(x))\in\mathcal{G}$, set $X_{,ij}:=\partial_i
\partial_j X-\Gamma_{ij}^{k}X_k$ with $\Gamma_{ij}^{k}$ the
Christoffel symbols of the metric on $\mathcal{G}$. Then it is easy
to know
\begin{eqnarray*} \label{sff-1}
h_{ij}=-\left\langle X_{,ij}, \nu\right\rangle_{L},
 \end{eqnarray*}
 and have the following identities
\begin{eqnarray} \label{gf-1}
&&(\mathrm{Gauss~formula}) \qquad \qquad \qquad X_{,ij}=h_{ij}\nu,
\label{Gauss
for} \\
&& (\mathrm{Weingarten~formula}) \qquad \qquad \nu_{,i}=h_{ij}X^{j}.
\label{Wein for}
\end{eqnarray}
By \cite[Section 2]{GaoY2}, we have
\begin{equation}\label{Gauss-1}
R_{ijkl}=h_{il}h_{jk}-h_{ik}h_{jl},
\end{equation}
\begin{equation}\label{Codazzi-1}
\nabla_{k}h_{ij}=\nabla_{j}h_{ik}, \qquad (i.e.,~h_{ij,k}=h_{ik,j})
\end{equation}
and
\begin{eqnarray}\label{Laplace}
\Delta h_{ij}=H_{,ij}-Hh_{ik}h^{k}_{j}+h_{ij}|A|^{2}.
\end{eqnarray}
As usual, here the comma ``," in subscript of a given tensor means
doing covariant derivatives. Besides, we make an agreement that, for
simplicity, in the sequel the comma ``," in subscripts will be
omitted unless necessary. BTW, formulae (\ref{sff-1})-(\ref{Laplace}) have also been mentioned in our previous works \cite{gm2,gm21,gm3,GLM}. \\

For any equation
\begin{equation}\label{F's equations}
F(A)=f(\lambda_{1},\lambda_{2},\ldots,\lambda_{n}),
\end{equation}
where $A$ is the second fundamental form of the spacelike graphic
hypersurface $\mathcal{G}\subset\mathbb{R}^{n+1}_{1}$. We can prove
the following two conclusions:

\begin{proposition}\label{elliptic}
Let $\mathcal{G}$, defined by (\ref{G1}), be a smooth $k$-admissible
spacelike graphic hypersurface in $\mathbb{R}^{n+1}_{1}$, $0\leq
l\leq k-2$, $2\leq k\leq n$. Then the operator\footnote{~Here, for
accuracy, $\frac{\sigma_{k}}{\sigma_{l}}[u]$ should be
$\frac{\sigma_{k}}{\sigma_{l}}(\lambda(A(u)))$. We write as
$\frac{\sigma_{k}}{\sigma_{l}}[u]$ for the purpose of simplifying
and emphasizing that $A$ and its eigenvalues depend on the graphic
function $u$. Based on this fact, if necessary, sometimes we also
write $\frac{\sigma_{k}}{\sigma_{l}}(\lambda(A(u)))$ (or
$\frac{\sigma_{k}}{\sigma_{l}}$) as
$\frac{\sigma_{k}}{\sigma_{l}}[u]$ to emphasize this connection.
This simplification will be used similarly in the sequel.}
$$\left(\frac{\sigma_{k}}{\sigma_{l}}[u]\right)^{\frac{1}{k-l}}$$
is elliptic.
\end{proposition}

\begin{proof}
To prove the ellipticity of the operator $\sigma_{k}/\sigma_{l}$, it
is equivalent to prove
$$\frac{\partial}{\partial\lambda_{i}}\left(\frac{\sigma_{k}}{\sigma_{l}}[u]\right)>0 \qquad \mathrm{for~~all}~~i=1,2,\cdots,n,$$
where $\lambda_{i}'s$ are the principal curvatures of $\mathcal{G}$.

By direct calculation, we have
\begin{equation*}
\begin{split}
\frac{\partial}{\partial\lambda_{i}}\left(\frac{\sigma_{k}}{\sigma_{l}}[u]\right)&=\frac{\sigma_{k-1}(\lambda|i)\sigma_{l}-\sigma_{k}\sigma_{l-1}(\lambda|i)}{\sigma_{l}^{2}} \\
&=
\frac{\sigma_{k-1}(\lambda|i)\sigma_{l}(\lambda|i)-\sigma_{k}(\lambda|i)\sigma_{l-1}(\lambda|i)}{\sigma_{l}^{2}}.
\end{split}
\end{equation*}
If $\sigma_{k}(\lambda|i)<0$, this proposition follows. If
$\sigma_{k}(\lambda|i) > 0$, then one has $\sigma_{m}(\lambda|i)>0$,
$m=1,2,\cdots,k-1$. So, using the generalized Newton-Maclaurin
inequality (see, e.g., \cite{mt1,t2}), we have
\begin{equation*}
\frac{\sigma_{k}(\lambda|i)/C_{n}^{k}}{\sigma_{l}(\lambda|i)/C_{n}^{l}}
\leq
\frac{\sigma_{k-1}(\lambda|i)/C_{n}^{k-1}}{\sigma_{l-1}(\lambda|i)/C_{n}^{l-1}},
\end{equation*}
which implies
\begin{equation*}
\frac{\sigma_{k}(\lambda|i)}{\sigma_{l}(\lambda|i)} <
\frac{\sigma_{k}(\lambda|i)}{\sigma_{l}(\lambda|i)}\cdot\frac{k}{l}\cdot\frac{n-l+1}{n-k+1} \leq \frac{\sigma_{k-1}(\lambda|i)}{\sigma_{l-1}(\lambda|i)}.
\end{equation*}
Then we have
$$\sigma_{k-1}(\lambda|i)\sigma_{l}(\lambda|i)-\sigma_{k}(\lambda|i)\sigma_{l-1}(\lambda|i) >0, $$
and the ellipticity of the operator $\sigma_{k}/\sigma_{l}$ follows
directly.
\end{proof}

By \cite[Lemmas 2.2 and 2.3]{GLM}, we have:

\begin{lemma}\label{g's formula}
For the function $F$ defined by \eqref{F's equations} and the quantity $\vartheta$ given in the PCP \eqref{main equation}, one has
\begin{equation*}\label{Codazzi-1}
F^{ij}\nabla_{i}\nabla_{j}\nu = \nu F^{ij}h^{m}_{j}h_{im}+F^{ij}\nabla_{i}h^{m}_{j}X_{m},
\end{equation*}
\begin{equation*}\label{Codazzi-1}
\Delta\vartheta=\sigma_{1}+\nabla^{i}\sigma_{1}\langle X,X_{i}\rangle_{L}+|A|^{2}\vartheta,
\end{equation*}
and
\begin{equation*}\label{F's equation}
F^{ij}\nabla_{i}\nabla_{j}\sigma_{1} = -F^{ij,pq}\nabla^{k}h_{ij}\nabla_{k}h_{pq}+F^{ij}h^{m}_{j}h_{im}\sigma_{1}-F^{ij}h_{ij}|A|^{2}+\Delta f,
\end{equation*}
\begin{equation*}\label{F's equation}
F^{ij}\nabla_{i}\nabla_{j}h_{mn} = -F^{ij,pq}\nabla_{n}h_{ij}\nabla_{m}h_{pq}+F^{ij}h^{l}_{j}h_{il}h_{mn}-F^{ij}h^{l}_{m}h_{ln}h_{ij}+
\nabla_{m}\nabla_{n}f,
\end{equation*}
where $F_{ij}:=\partial F/\partial h_{ij}$,
$F^{ij,pq}:=\partial^{2}F/\partial h_{ij}\partial h_{pq}$.
\end{lemma}

\section{$C^{1}$ estimate} \label{c1es}
\subsection{Boundary estimate}
 In $\mathbb{R}^{n+1}_{1}$, for any $k\in\{1,\cdots,n\}$, we assume the existence and uniqueness of solutions to the prescribed $k$-th Weingarten
 curvature problem with DBC. Let $s^{+}$
be the solution of the following Dirichlet problem
\begin{equation*}
\left\{
\begin{aligned}
&\sigma_{2}[s] =
C^{2}_{n}\left(\frac{C^{l}_{n}}{C^{k}_{n}}\psi(x,s,\vartheta)\right)^{\frac{2}{k-l}}
\qquad &&x\in
M^{n}\subset\mathscr{H}^{n}(1)\subset\mathbb{R}^{n+1}_{1},
\\
&s = \varphi \qquad&&x\in \partial M^{n}.
\end{aligned}
\right.
\end{equation*}
From the generalized Newton-Maclaurin inequality, we have
\begin{equation*}
\sigma_{2}[s^{+}]\leq\sigma_{2}[u].
\end{equation*}
By the comparison principle, we have $u\leq s^{+}$ in $M^{n}$, and
thus $\frac{\partial u}{\partial\nu}\geq \frac{\partial
s^{+}}{\partial\nu}$. In order to get a lower barrier, let $s^{-}$
be the solution of the following Dirichlet problem
\begin{equation*}\label{main equations}
\left\{
\begin{aligned}
&\sigma_{k-1}[s] = \psi(x,s,\vartheta)^{\frac{k-1}{k-l}}\cdot
C^{\frac{k-1}{l-k+1}}(n,k,l)\qquad &&x\in
M^{n}\subset\mathscr{H}^{n}(1)\subset\mathbb{R}^{n+1}_{1},
\\
&s = \varphi\qquad&&x\in \partial M^{n},
\end{aligned}
\right.
\end{equation*}
with
$C(n,k,l):=(C^{l}_{n})^{\frac{1}{k-l}-1}(C^{k}_{n})^{-\frac{1}{k-l}+1}(C^{k-1}_{n})^{\frac{l}{k-1}-1}$.
Also from the generalized Newton-Maclaurin inequality, we have
\begin{equation*}
\sigma_{k-1}[s^{-}]\geq\sigma_{k-1}[u].
\end{equation*}
So $s^{-}\leq u$ in $M^{n}$, and thus $\frac{\partial u}{\partial\nu}\leq\frac{\partial s^{-}}{\partial\nu}$.

\begin{remark}
\rm{The existence and uniqueness of solutions to the prescribed
$2$-th Weingarten curvature problem with DBC have been shown in
\cite{GLM}, we can only take $k=2$, $l=0$ here. So, for the first
equations, one has
\begin{equation*}
\left\{
\begin{aligned}
&\sigma_{2}[s] =\psi(x,s,\vartheta)
\qquad &&x\in
M^{n}\subset\mathscr{H}^{n}(1)\subset\mathbb{R}^{n+1}_{1},
\\
&s = \varphi \qquad&&x\in \partial M^{n},
\end{aligned}
\right.
\end{equation*}
while, for the second equations, one has
\begin{equation*}
\left\{
\begin{aligned}
&\sigma_{1}[s] =
\psi^{\frac{1}{2}}(x,s,\vartheta)\cdot n\left( \frac{n(n-1)}{2} \right)^{-\frac{1}{2}}
\qquad &&x\in
M^{n}\subset\mathscr{H}^{n}(1)\subset\mathbb{R}^{n+1}_{1},
\\
&s = \varphi \qquad&&x\in \partial M^{n}.
\end{aligned}
\right.
\end{equation*}
}
\end{remark}

\subsection{Maximum principle}
The upper bound on $Du$ amounts to an upper bound on
$\mathcal{W}:=\frac{1}{v}=\frac{1}{\sqrt{1-|D\pi|^{2}}}$, where $\pi:=\ln u$. Therefore, it would follow from the boundary estimate once one can prove that $\mathcal{W}e^{S\pi}$ cannot attain an interior maximum for $S$
sufficiently large under control.

\begin{proposition}\label{maximum principle}
Let $u$ be the admissible solution of the PCP \eqref{main equation}. Then
\begin{equation*}
\sup\limits_{\overline{M^{n}}}\mathcal{W} \leq
\left(\sup\limits_{\partial M^{n}}\mathcal{W}
\right)e^{S_{2}\left(2\sup\limits_{\partial M^{n}}|\varphi| +
\mathrm{diam}(M^{n})\right)},
\end{equation*}
 where as usual $\mathrm{diam}(M^{n})$ stands for the diameter of
 the bounded domain $M^{n}\subset\mathscr{H}^{n}(1)$.
\end{proposition}

\begin{proof}
By contradiction, suppose that
$\sup_{\overline{M^{n}}}\mathcal{W}e^{S\pi}$ is achieved at an
interior point $x_{0}\in M^{n}$. At $x_0$, we choose a nice basis
for the convenience of computations, that is, let $\{e_{1}, e_{2},
\cdots, e_{n}\}$ be an orthonormal basis of $T_{x_0}M^{n}$ (i.e.,
the tangent space at $x_{0}$ diffeomorphic to $\mathbb{R}^{n}$) such
that $D\pi(x_{0}) = |D\pi(x_{0})|e_{1}$, and moreover, the matrix
$\left((D^{2}\pi(x_0))_{ij}\right)_{(n-1)\times(n-1)}$, $2\leq
i,j\leq n$, is orthogonal under the basis $\{e_{2}, \cdots,
e_{n}\}$. Since $|\pi_{1}| \leq |D\pi|$ on $\overline{M^{n}}$ and
$\pi_{1}(x_{0}) = |D\pi(x_{0})|$. The function
$$\ln\left( \frac{1}{\sqrt{1-\pi_{1}^{2}}}\right) + S\pi = -\frac{1}{2}\ln\left( 1-\pi_{1}^{2} \right) + S\pi $$
has a maximum at $x_{0}$ as well. Hence, at $x_{0}$, for any $i\in\{
1,\cdots,n \}$, one has
$$\frac{\pi_{1i}\pi_{1}}{1-\pi_{1}^{2}} + S\pi_{i} = 0.$$
So the matrix of the curvature operator is diagonal, with diagonal
entries $(\frac{1}{uv}(1+\frac{\pi_{11}}{v^{2}}),
\frac{1}{uv}(1+\pi_{22}), \cdots, \frac{1}{uv}(1+\pi_{nn}) )$.
Moreover, still at $x_{0}$, one has $\pi_{111}\pi_{1} \leq
-\pi_{11}^{2} - \frac{2(\pi_{1}\pi_{11})^{2}}{1-\pi_{1}^{2}} -
S\pi_{11}(1-\pi_{1}^{2})$, and for $i\geq2$, $\pi_{1ii}\pi_{1} \leq
-(1-\pi_{1}^{2})S\pi_{ii}$. Then we have
\begin{equation*}
\sum_{i=1}^{n}\left(\frac{\sigma_{k}}{\sigma_{l}}\right)_{\lambda_{i}}\cdot\lambda_{i,1} =
\sum_{i=1}^{n}\left(\frac{\sigma_{k}}{\sigma_{l}}\right)_{\lambda_{i}}\cdot h^{i}_{i,1} = \psi_{1}.
\end{equation*}
Since $h_{i}^{i} = \frac{1}{uv}\left(1+(\sigma^{ik} +
\frac{\pi^{i}\pi^{k}}{v^{2}})\pi_{ik}\right)$, we have
\begin{equation*}
h_{1,1}^{1} = \frac{\pi_{1}(3S^{2}-1)}{uv} +\frac{\pi_{111}}{uv^{3}},
\end{equation*}
\begin{equation*}
\begin{aligned}
h_{i,1}^{i} &= \frac{\pi_{ii1}}{uv} - \frac{\pi_{1}\pi_{ii}(S+1)}{uv} -
\frac{\pi_{1}(S+1)}{uv} \qquad \mathrm{for}~~ i\geq2.
\end{aligned}
\end{equation*}
The differentiated equation, multiplied by $\pi_{1}$, becomes:
\begin{equation*}
\begin{aligned}
&(\frac{\sigma_{k}}{\sigma_{l}})_{\lambda_{1}}\left( \frac{\pi_{1}(3S^{2}-1)}{uv} + \frac{\pi_{111}}{uv^{3}} \right)\\
&\quad
+ \sum_{i\geq2}(\frac{\sigma_{k}}{\sigma_{l}})_{\lambda_{i}}\left(
\frac{\pi_{ii1}}{uv} - \frac{\pi_{1}\pi_{ii}(S+1)}{uv} -
\frac{\pi_{1}(S+1)}{uv} \right) = \pi_{1}\psi_{1}.
\end{aligned}
\end{equation*}
From the maximum conditions, we have
\begin{equation*}
\frac{\pi_{1}^{2}(3S^{2}-1)}{uv} + \frac{\pi_{111}\pi_{1}}{uv^{3}}
\leq \frac{\pi_{1}^{2}(S^{2}-1)}{uv},
\end{equation*}
and, since $\pi_{1ii} = \pi_{ii1} - \pi_{1}$, we have
\begin{equation*}
\begin{aligned}
&\quad
\frac{\pi_{ii1}\pi_{1}}{uv} - \frac{\pi_{1}^{2}\pi_{ii}(S+1)}{uv} - \frac{\pi_{1}^{2}(S+1)}{uv}\\
&\leq -\frac{1}{u}vS\pi_{ii} - \frac{\pi_{1}^{2}S}{uv} -
\frac{\pi_{1}^{2}\pi_{ii}(S+1)}{uv}.
\end{aligned}
\end{equation*}
Then we can infer
\begin{equation*}
(\frac{\sigma_{k}}{\sigma_{l}})_{\lambda_{1}}\cdot\frac{\pi_{1}^{2}(S^{2}-1)}{uv}
- \sum_{i\geq 2}(\frac{\sigma_{k}}{\sigma_{l}})_{\lambda_{i}}\left(
\frac{1}{u}vS\pi_{ii} + \frac{\pi_{1}^{2}S}{uv} +
\frac{\pi_{1}^{2}\pi_{ii}(S+1)}{uv} \right) \geq \pi_{1}\psi_{1},
\end{equation*}
so we have
\begin{equation*}
\begin{aligned}
&
-(k-l)v^{2}S\frac{\sigma_{k}}{\sigma_{l}}-(k-l)(S+1)\pi_{1}^{2}\frac{\sigma_{k}}{\sigma_{l}}\\
&\quad
+\frac{v}{u}\mathcal{M}S+\frac
{1}{uv}\mathcal{M}\pi_{1}^{2}-(\frac{\sigma_{k}}{\sigma_{l}})_{\lambda_{1}}\frac{1}{uv}(v^{2}S^{2}+
\pi_{1}^{2})\geq\pi_{1}\psi_{1},
\end{aligned}
\end{equation*}
where $\mathcal{M}:=\sum_{i=1}^{n}(\frac{\sigma_{k}}{\sigma_{l}})_{\lambda_{i}}$, and then, since $(\frac{\sigma_{k}}{\sigma_{l}})_{\lambda_{i}}>0$, there exists a positive constant $m>1$ such that $(\frac{\sigma_{k}}{\sigma_{l}})_{\lambda_{1}} = \frac{1}{m}\mathcal{M}$.
Then, we have
\begin{equation*}
(k-l)\psi S\leq \mathcal{M}\left(\frac{v}{u}S+\frac{\pi_{1}^{2}}{uv}-\frac{1}{m}\frac{v}{u}S(S-1)\right)-\pi_{1}\psi_{1}.
\end{equation*}
We want
\begin{equation*}
\frac{v}{u}S+\frac{\pi_{1}^{2}}{uv}-\frac{1}{m}\frac{v}{u}S(S-1)\leq 0,
\end{equation*}
which is equivalent to
\begin{equation*}
v^{2}S+\pi_{1}^{2}\leq\frac{1}{m}v^{2}S(S-1).
\end{equation*}
Since $\pi_{1}^{2}\leq\rho^{2}<1$, choosing $S=S_{1}$ large enough
such that
\begin{equation*}
\rho^{2}\leq 1-\frac{m(S+1)}{S(S-1)}.
\end{equation*}
So we have
$$(k-l)\psi S\leq \sup\limits_{\overline{M^{n}}}|D\psi|.$$
Then choosing $S_{2}>\max\left\{\frac{\sup\limits_{\overline{M^{n}}}|D\psi|}{(k-l)\inf\limits_{\overline{M^{n}}}\psi},
S_{1}\right\}$, we reach a contradiction.
\end{proof}

\section{Curvature maximum principle} \label{S4}

We write \eqref{main equation} in the form
\begin{equation}\label{3.1}
F(A)=\left(\frac{\sigma_{k}}{\sigma_{l}}\right)^{\frac{1}{k-l}}(A)=\psi^{\frac{1}{k-l}}(X,\vartheta)=f(X,\vartheta)
\qquad \mathrm{for ~ any} ~ X\in\mathcal{G}.
\end{equation}

\begin{proof}[Proof of Theorem \ref{main1.1}]
 Consider the function
\begin{eqnarray*}
W(A)=\sigma_{1}(A),
\end{eqnarray*}
which attains its maximum value at some
$X_{0}=(x_{0},u(x_{0}))\in\mathcal{G}$. If $x_{0}\in\partial M^{n}$,
then our claim \eqref{A's boundary} follows directly. Now, we try to
prove this claim in the case that $x_{0}\notin\partial M^{n}$.
Choose the frame fields  $e_{1},e_{2},\cdots,e_{n}, \nu$ at $X_{0}$
such that $e_{1},e_{2},\cdots,e_{n}\in T_{X_{0}}\mathcal{G}$ at
$X_{0}$ and $(h_{ij})_{n\times n}$ is diagonal at $X_{0}$ with
eigenvalues $h_{11}\geq h_{22}\geq\cdots\geq h_{nn}$. Here, as
usual, $T_{X_{0}}\mathcal{G}$ denotes the tangent space of the
graphic hypersurface $\mathcal{G}$ at $X_{0}$. For each
$i=1,\ldots,n$, we have
\begin{equation*}\label{3.2}
\nabla_{i}\sigma_{1}=0 \qquad \mathrm{at}~ X_{0}.
\end{equation*}
Therefore, at $X_{0}$, it follows that
\begin{equation}\label{3.3}
\begin{split}
0&\geq F^{ij}\nabla_{i}\nabla_{j}\sigma_{1}\\
&=-F^{ij,pq}\nabla_{l}h_{ij}\nabla_{l}h_{pq}+F^{ij}h_{im}h_{mj}\sigma_{1}-F^{ij}h_{ij}|A|^{2}+\Delta
f.
\end{split}
\end{equation}
Since $f$ is convex in $\vartheta$, together with Lemma
\ref{g's formula}, we have
\begin{equation}\label{3.4}
\begin{split}
\Delta f&=\frac{\partial^2 f}{\partial X^{\alpha}\partial X^{\beta}}\nabla_{l}X^{\alpha}\nabla_{l}X^{\beta}+2\frac{\partial^2 f}{\partial X^{\alpha}\partial \vartheta}\nabla_{l}X^{\alpha}\nabla_{l}\vartheta\\
&\quad
+\frac{\partial^2 f}{\partial \vartheta^{2}}|\nabla \vartheta|^{2}+\frac{\partial f}{\partial X^{\alpha}}\Delta X^{\alpha}+\frac{\partial f}{\partial \vartheta}\Delta \vartheta\\
&\geq\frac{\partial f}{\partial \vartheta}\Delta \vartheta+\frac{\partial^2 f}{\partial \vartheta^{2}}|\nabla \vartheta|^{2}-c_{1}\sigma_{1}-c_{2}\\
&\geq\frac{\partial f}{\partial
\vartheta}\vartheta|A|^{2}-c_{1}\sigma_{1}-c_{2},
\end{split}
\end{equation}
where positive constants $c_{1}$, $c_{2}$ depend on
$||\varphi||_{C^{1}(\overline{M^{n}})}$,
$||\psi||_{C^{2}\left(\overline{M^{n}}\times\left[\inf\limits_{\partial
M^{n}}u,\sup\limits_{\partial
M^{n}}u\right]\times\mathbb{R}\right)}$, and $X^{\alpha}:=\langle X,
\partial_{\alpha}\rangle_{L}$, $\alpha=1,2,\cdots,n+1$. Obviously,
$\partial_{1},\partial_{2},\cdots,\partial_{n}$ are the
corresponding coordinate vector fields on $\mathscr{H}^{n}(1)$,
$\partial_{n+1}:=\partial_{r}$. Putting \eqref{3.4} into \eqref{3.3}
yields
\begin{equation}\label{3.5}
\begin{split}
0&\geq F^{ij}\nabla_{i}\nabla_{j}\sigma_{1}\\
&\geq -F^{ij,pq}\nabla_{l}h_{ij}\nabla_{l}h_{pq}+F^{ij}h_{im}h_{mj}\sigma_{1}\\
&\quad
+(\frac{\partial f}{\partial \vartheta}\cdot\vartheta-f)|A|^{2}-c_{1}\sigma_{1}-c_{2}\\
&\geq F^{ij}h_{im}h_{mj}\sigma_{1}-c_{1}\sigma_{1}-c_{2},
\end{split}
\end{equation}
where we have used \eqref{sf's conditions} and the concavity of $F$.
On the other hand, using the properties of Garding cone and the
Cauchy inequality, we have
\begin{equation}\label{3.6}
\begin{split}
F^{ij}h_{im}h_{mj}&=\sum_{i=1}^{n}\frac{\partial}{\partial\lambda_{i}}\left[(\frac{\sigma_{k}}{\sigma_{l}})^{\frac{1}{k-l}}\right]\lambda_{i}^{2}\\
&\geq
n\left[(\frac{\sigma_{k}}{\sigma_{l}})^{\frac{1}{k-l}}\right]_{\lambda_{\mathrm{min}}}\sum_{i=1}^{n}\lambda_{i}^{2}\\
&\geq
\frac{n}{m^{'}}\left(\frac{C^{k}_{n}}{C^{l}_{n}}\right)^{\frac{1}{k-l}}\sum_{i=1}^{n}\lambda_{i}^{2}
\geq\frac{1}{m^{'}}\left(\frac{C^{k}_{n}}{C^{l}_{n}}\right)^{\frac{1}{k-l}}\sigma_{1}^{2},
\end{split}
\end{equation}
where
$\left[(\frac{\sigma_{k}}{\sigma_{l}})^{\frac{1}{k-l}}\right]_{\lambda_{\mathrm{min}}}:=\min\left\{\frac{\partial}{\partial\lambda_{i}}\left[(\frac{\sigma_{k}}{\sigma_{l}})^{\frac{1}{k-l}}\right]\right\}$,
$i=1,2,\cdots,n$. Since
$\left(\frac{\sigma_{k}}{\sigma_{l}}\right)_{\lambda_{i}} > 0$,
there exists a positive constant $m'$ such that
$\left[(\frac{\sigma_{k}}{\sigma_{l}})^{\frac{1}{k-l}}\right]_{\lambda_{\mathrm{min}}}
=
\frac{1}{m'}\sum_{i=1}^{n}\frac{\partial}{\partial\lambda_{i}}\left[(\frac{\sigma_{k}}{\sigma_{l}})^{\frac{1}{k-l}}\right]$.
Taking \eqref{3.6} into \eqref{3.5}, it is easy to know that
 $\sigma_{1}$ is bounded. Then the conclusion of
Theorem \ref{main1.1}, i.e. \eqref{A's boundary}, follows naturally.
\end{proof}

\section{Curvature estimates} \label{S5}
Let
\begin{equation*}
\mathcal{P}(\lambda):= F(A) =
\left(\frac{\sigma_{k}}{\sigma_{l}}\right)^{\frac{1}{k-l}}(A) =
f(X,\vartheta) \qquad \mathrm{for ~ any} ~ X\in\mathcal{G}.
\end{equation*}
Set
\begin{equation}\label{3.1}
\left(\frac{\sigma_{k}}{\sigma_{l}}\right)^{\frac{1}{k-l}}(\lambda_{1},\lambda_{2},\cdots\lambda_{n})
= \mathcal{P}(\lambda_{1},\lambda_{2},\cdots\lambda_{n}),
\end{equation}
\begin{equation}\label{3.1}
\mathrm{tr}F^{ij} = \sum_{i=1}^{n} F^{ii},\qquad \mathcal{P}_{i} =
\frac{\partial \mathcal{P}}{\partial\lambda_{i}}.
\end{equation}

First, we list a useful lemma (see, e.g., \cite{ju2}).
\begin{lemma}\label{F_{ij}'s equation 1}
For any symmetric matrix $\eta = (\eta_{ij})$, we have
\begin{equation}\label{F's equation a}
F^{ij,pq}\eta_{ij}\eta_{pq} = \sum_{i=1}^{n}\frac{\partial^{2}\mathcal{P}}{\partial\lambda_{i}\partial\lambda_{j}}\eta_{ii}\eta_{jj}+
\sum_{i\neq j}\frac{\mathcal{P}_{i}-\mathcal{P}_{j}}{\lambda_{i}-\lambda_{j}}\eta_{ij}^{2}.
\end{equation}
The second term on RHS of \eqref{F's equation a} is nonpositive if $\mathcal{P}$ is concave, and it is interpreted as the limit if $\lambda_{i} = \lambda_{j}$.
\end{lemma}

\begin{proof} [Proof of Theorem \ref{main1.2}] Using a similar
argument to that in the proof of \cite[Theorem 1.3]{GLM}.

 Let $\eta = \varphi-u$, as observed at the beginning of the proof
of \cite[Theorem 1.3]{GLM}, one knows that $\eta>0$ in $M^{n}$. We
now consider the function
\begin{equation*}
G = \eta^\alpha e^{\Psi(\vartheta)}h_{ij}\tau_{i}\tau_{j},
\end{equation*}
achieving its maximum value at some $X_{0}\in\mathcal{G}$, where
$\alpha\geq1$, $\Psi$ is a function determined later and satisfies
$\Psi^{'}:=\frac{\partial\Psi}{\partial\vartheta}\geq0$. Without
loss of genenality, one may choose the frame field $e_{1} =
\tau,e_{2},\cdots,e_{n},\nu$ such that $e_{1},e_{2},\cdots,e_{n}\in
T_{X_{0}}\mathcal{G}, \nabla_{e_{i}}e_{j} = 0$ at $X_0$ for all $i =
1,2,\cdots,n$, and $(h_{ij})_{n\times n}$ is diagonal at $X_0$ with
eigenvalues $h_{11}\geq h_{22}\geq \cdots h_{nn}$. At $X_0$, for
each $i = 1,2,\cdots,n$, one has
\begin{equation}\label{5.4}
\alpha\frac{\nabla_{i}\eta}{\eta}+\Psi^{'}\nabla_{i}\vartheta+\frac{\nabla_{i}h_{11}}{h_{11}}=0,
\end{equation}
\begin{equation*}
\begin{split}
\alpha(\frac{\nabla_{i}\nabla_{j}\eta}{\eta}-\frac{\nabla_{i}\eta\nabla_{j}\eta}{\eta^{2}})
&+\Psi^{''}\nabla_{i}\vartheta\nabla_{j}\vartheta\\
&+\Psi^{'}\nabla_{i}\nabla_{j}\vartheta
+\frac{\nabla_{i}\nabla_{j}h_{11}}{h_{11}}
-\frac{\nabla_{i}h_{11}\nabla_{j}h_{11}}{h_{11}^{2}}\leq0.
\end{split}
\end{equation*}
Therefore, by Lemma \ref{g's formula}, we have
\begin{equation*}
\begin{split}
0&\geq\alpha F^{ij}(\frac{\nabla_{i}\nabla_{j}\eta}{\eta}-\frac{\nabla_{i}\eta\nabla_{j}\eta}{\eta^{2}})
+\Psi^{''}F^{ij}\nabla_{i}\vartheta\nabla_{j}\vartheta+\Psi^{'}F^{ij}\nabla_{i}\nabla_{j}\vartheta\\
&\quad
+F^{ij}\frac{\nabla_{i}\nabla_{j}h_{11}}{h_{11}}-F^{ij}\frac{\nabla_{i}h_{11}\nabla_{j}h_{11}}{h_{11}^{2}}\\
&=\alpha F^{ij}(\frac{\nabla_{i}\nabla_{j}\eta}{\eta}-\frac{\nabla_{i}\eta\nabla_{j}\eta}{\eta^{2}})
+\Psi^{''}F^{ij}\nabla_{i}\vartheta\nabla_{j}\vartheta+\Psi^{'}F^{ij}\nabla_{i}\nabla_{j}\vartheta\\
&\quad
-fh_{11}
+F^{ij}h_{im}h_{jm}+\frac{\nabla_{1}\nabla_{1}f}{h_{11}}
-\frac{1}{h_{11}}F^{ij,pq}\nabla_{1}h_{ij}\nabla_{1}h_{pq}
-F^{ij}\frac{\nabla_{i}h_{11}\nabla_{j}h_{11}}{h_{11}^{2}}.
\end{split}
\end{equation*}
We also find that
\begin{equation*}
F^{ij}\nabla_{i}\nabla_{j}\vartheta = \vartheta F^{ij}h_{im}h_{jm}+f+\nabla_{l} f\langle X,X_{l}\rangle_{L}.
\end{equation*}
Consequently,
\begin{equation}\label{5.5}
\begin{split}
0&\geq\alpha F^{ij}(\frac{\nabla_{i}\nabla_{j}\eta}{\eta}-\frac{\nabla_{i}\eta\nabla_{j}\eta}{\eta^{2}})
+\Psi^{''}F^{ij}\nabla_{i}\vartheta\nabla_{j}\vartheta+\Psi^{'}\nabla_{l}f\langle X,X_{l}\rangle_{L}-fh_{11}\\
&\quad
+(\Psi^{'}\vartheta+1)F^{ij}h_{im}h_{jm}+\frac{\nabla_{1}\nabla_{1}f}{h_{11}}-\frac{1}{h_{11}}F^{ij,pq}\nabla_{1}h_{ij}\nabla_{1}h_{pq}
-F^{ij}\frac{\nabla_{i}h_{11}\nabla_{j}h_{11}}{h_{11}^{2}}.
\end{split}
\end{equation}
Since $f$ is convex in $\vartheta$, we have
\begin{equation*}
\begin{split}
\nabla_{1}\nabla_{l}f &=
\frac{\partial^{2}f}{\partial X^{\alpha}\partial X^{\beta}}\nabla_{1}X^{\alpha}\nabla_{1} X^{\beta}
+2\frac{\partial^{2}f}{\partial X^{\alpha}\partial X^{\beta}}
\nabla_{1}X^{\alpha}\nabla_{1}\vartheta
+\frac{\partial^{2}f}{\partial\vartheta^{2}}|\nabla_{1}\vartheta|^{2}\\
&\quad
+\frac{\partial f}{\partial X^{\alpha}}\nabla_{1}\nabla_{1}X^{\alpha}
+\frac{\partial f}{\partial\vartheta}\nabla_{1}\nabla_{1}\vartheta\\
&\geq\frac{\partial f}{\partial\vartheta}\nabla_{1}\nabla_{1}\vartheta-c_{3}h_{11}-c_{4}\\
&=\frac{\partial f}{\partial\vartheta}(\vartheta h_{11}^{2}+\nabla_{l}h_{11}\langle X,X_{l}\rangle_{L})-c_{3}h_{11}-c_{4},
\end{split}
\end{equation*}
where $c_{3}$, $c_{4}$ are positive constant depending on $||\varphi||_{C^{1}(\overline{M^{n}}
)}$, $||\psi||_{C^{2}\left(\overline{M^{n}}\times\left[\inf\limits_{\partial M^{n}}u,\sup\limits_{\partial M^{n}}u\right]\times\mathbb{R}\right)}$.
Inserting this into (\ref{5.5}) yields
\begin{equation}\label{5.6}
\begin{split}
0&\geq\alpha F^{ij}(\frac{\nabla_{i}\nabla_{j}\eta}{\eta}-\frac{\nabla_{i}\eta\nabla_{j}\eta}{\eta^{2}})
+\Psi^{''}F^{ij}\nabla_{i}\vartheta\nabla_{j}\vartheta+\Psi^{'}\nabla_{l}f\langle X,X_{l}\rangle_{L}\\
&\quad
+(\frac{\partial f}{\partial\vartheta}\cdot\vartheta-f)h_{11}+(\Psi^{'}\vartheta+1)F^{ij}h_{im}h_{jm}+\frac{\partial f}{\partial\vartheta}\frac{\nabla_{l}h_{11}\langle X,X_{l}\rangle_{L}}{h_{11}}\\
&\quad
-\frac{1}{h_{11}}F^{ij,pq}\nabla_{1}h_{ij}\nabla_{1}h_{pq}
-F^{ij}\frac{\nabla_{i}h_{11}\nabla_{j}h_{11}}{h_{11}^{2}}-c_{3},
\end{split}
\end{equation}
where we have assumed that $h_{11}$ is sufficiently large. Otherwise, the assertion of Theorem \ref{main1.2} holds.

Next, we assume that $\varphi$ has been extended to be constant in the $\partial_{r}$ direction. Therefore,
\begin{equation*}
\begin{split}
\nabla_{i}\nabla_{j}\eta&=\sum_{\alpha,\beta=1}^{n}\frac{\partial^{2}\varphi}{\partial X^{\alpha}\partial X^{\beta}}
\nabla_{i}X^{\alpha}\nabla_{j}X^{\beta}+\sum_{\alpha=1}^{n}\frac{\partial\varphi}{\partial X^{\alpha}}
\nabla_{i}\nabla_{j}X^{\alpha}-u_{ij}\\
&\geq\sum_{\alpha=1}^{n}\frac{\partial\varphi}{\partial X^{\alpha}}
\nu^{\alpha}h_{ij}-c_{5}h_{ij}v,
\end{split}
\end{equation*}
where $c_{5}>0$ depends on $||\varphi||_{C^{1}(\overline{M^{n}})}$ and we have again used Gaussian formula and the assumption that $\varphi$ is affine. Consequently,
\begin{equation}\label{5.7}
F^{ij}\nabla_{i}\nabla_{j}\eta\geq(\sum_{\alpha=1}^{n}\frac{\partial\varphi}{\partial X^{\alpha}}
\nu^{\alpha}-c_{5}v)F^{ij}h_{ij}\geq-c_{6},
\end{equation}
where positive constant $c_{6}$ depends on $c_{5}$, $||\psi||_{C^{0}\left(\overline{M^{n}}\times\left[\inf\limits_{\partial M^{n}}u,\sup\limits_{\partial M^{n}}u\right]\times\mathbb{R}\right)}$ and $||\varphi||_{C^{1}(\overline{M^{n}})}$.
Combining (\ref{5.6}) and (\ref{5.7}), at $X_{0}$, we have
\begin{equation}\label{5.8}
\begin{split}
0&\geq -\frac{c_{6}\alpha}{\eta}-\alpha F^{ij}\frac{\nabla_{i}\eta\nabla_{j}\eta}{\eta^{2}}+\Psi^{''}F^{ij}\nabla_{i}\vartheta\nabla_{j}\vartheta
+\Psi^{'}\nabla_{l}f\langle X,X_{_{l}}\rangle_{L}\\
&\quad
+(\frac{\partial f}{\partial\vartheta}\cdot\vartheta-f)h_{11}+(\Psi^{'}\vartheta+1)F^{ij}h_{im}h_{jm}
+\frac{\partial f}{\partial\vartheta}\frac{\nabla_{l}h_{11}\langle X,X_{l}\rangle_{L}}{h_{11}}\\
&\quad
-\frac{1}{h_{11}}F^{ij,pq}\nabla_{1}h_{ij}\nabla_{1}h_{pq}-F^{ij}\frac{\nabla_{i}h_{11}\nabla_{j}h_{11}}{h_{11}^{2}}-c_{3}.
\end{split}
\end{equation}
We now estimate the remaining terms in (\ref{5.8}), and divide the argument into two cases.

\textbf{Case 1}. Assume that there exists a positive constant $\zeta$ to be determined such that
\begin{equation}\label{5.9}
h_{nn}\leq -\zeta h_{11}.
\end{equation}
Using the critical point condition (\ref{5.4}), we have
\begin{equation*}
\begin{split}
F^{ij}\frac{\nabla_{i}h_{11}\nabla_{j}h_{11}}{h_{11}^{2}}&=F^{ij}(\alpha\frac{\nabla_{i}\eta}{\eta}
+\Psi^{'}\nabla_{i}\vartheta)(\alpha\frac{\nabla_{j}\eta}{\eta}+\Psi^{'}\nabla_{j}\vartheta)\\
&\leq(1+\varepsilon^{-1})\alpha^{2}F^{ij}\frac{\nabla_{i}\eta\nabla_{j}\eta}{\eta^{2}}+(1+\varepsilon)(\Psi^{'})^{2}F^{ij}
\nabla_{i}\vartheta\nabla_{j}\vartheta
\end{split}
\end{equation*}
for any $\varepsilon>0$. Since $|\nabla\eta|\leq c_{7}(\widetilde{M^{n}})$ , so
\begin{equation*}
F^{ij}\frac{\nabla_{i}\eta\nabla_{j}\eta}{\eta^{2}}\leq
c_{8}\frac{\mathrm{tr}F^{ij}}{\eta^{2}},
\end{equation*}
where $c_{8}$ depends on $c_{7}$. Therefore, at $X_{0}$, we have
\begin{equation}\label{5.10}
\begin{split}
0&\geq-\frac{c_{6}\alpha}{\eta}-c_{9}\left[\alpha+(1+\varepsilon^{-1})\alpha^{2}\right]\frac{trF^{ij}}{\eta^{2}}
+\left[\Psi^{''}-(1+\varepsilon)(\Psi^{'})^{2}\right]F^{ij}\nabla_{i}\vartheta\nabla_{j}\vartheta\\
&\quad
+(\frac{\partial f}{\partial\vartheta}\cdot\vartheta-f)h_{11}+(\Psi^{'}\vartheta+1)F^{ij}h_{im}h_{jm}-c_{3}\\
&\quad
+\frac{\partial f}{\partial\vartheta}\frac{\nabla_{l}h_{11}\langle X,X_{l}\rangle_{L}}{h_{11}}
+\Psi^{'}\nabla_{l}f\langle X,X_{l}\rangle_{L},
\end{split}
\end{equation}
where $c_{9}:=\max\{1,c_{8}\}$ and the concavity of $F(A)$ has been
used. On the other hand, from (\ref{5.4}), the last two terms of the
RHS of (\ref{5.10}) are bounded from below
\begin{equation*}
\begin{split}
&\qquad\frac{\partial f}{\partial\vartheta}\frac{\nabla_{l}h_{11}\langle X,X_{l}\rangle_{L}}{h_{11}}
+\Psi^{'}\nabla_{l}f\langle X,X_{l}\rangle_{L}\\
&=(\Psi^{'}\nabla_{l}f-\alpha\frac{\partial f}{\partial\vartheta}\frac{\nabla_{l}\eta}{\eta}-
\frac{\partial f}{\partial\vartheta}\Psi^{'}\nabla_{l}\vartheta)\langle X,X_{l}\rangle_{L}\\
&=(\Psi^{'}\frac{\partial f}{\partial X^{\beta}}\nabla_{l}X^{\beta}
-\alpha\frac{\partial f}{\partial\vartheta}\frac{\nabla_{l}\eta}{\eta}
-\frac{\partial f}{\partial\vartheta}\Psi^{'}\nabla_{l}\vartheta)\langle X,X_{l}\rangle_{L}\\
&\geq -\frac{c_{10}\alpha}{\eta}-c_{11},
\end{split}
\end{equation*}
where $c_{10}$ is a positive constant depending on $c_{7}$, $||\varphi||_{C^{1}(\overline{M^{n}})}$, $||\psi||_{C^{1}\left(\overline{M^{n}}\times\left[\inf\limits_{\partial M^{n}}u,\sup\limits_{\partial M^{n}}u\right]\times\mathbb{R}\right)}$, and $c_{11}>0$ depends on $||\varphi||_{C^{1}(\overline{M^{n}})}$, $||\psi||_{C^{2}\left(\overline{M^{n}}\times\left[\inf\limits_{\partial M^{n}}u,\sup\limits_{\partial M^{n}}u\right]\times\mathbb{R}\right)}$. Therefore
\begin{equation}\label{5.11}
\begin{split}
0&\geq -\frac{c_{12}\alpha}{\eta}-c_{9}\left[\alpha+(1+\varepsilon^{-1})\alpha^{2}\right]\frac{trF^{ij}}{\eta^{2}}
+\left[\Psi^{''}-(1+\varepsilon)(\Psi^{'})^{2}\right]F^{ij}\nabla_{i}\vartheta\nabla_{j}\vartheta\\
&\quad
+(\frac{\partial f}{\partial\vartheta}\cdot\vartheta-f)h_{11}+(\Psi^{'}\vartheta+1)F^{ij}h_{im}h_{jm}-c_{13},
\end{split}
\end{equation}
where constant $c_{12}>0$ depends on $c_{6}$, $c_{10}$, and constant $c_{13}>0$ depends on $c_{3}$ and $c_{11}$. By the Weingarten formula, it follows that
\begin{equation*}
F^{ij}\nabla_{i}\vartheta\nabla_{j}\vartheta=F^{ij}h_{il}h_{jk}\langle X, X_{l}\rangle_{L}\langle X, X_{k}\rangle_{L}
\leq c_{14}F^{ij}h_{il}h_{jk},
\end{equation*}
where $c_{14}$ is a positive constant depending on $||\varphi||_{C^{1}(\overline{M^{n}})}$, and then we can take a function $\Psi$ satisfying
\begin{equation}\label{5.12}
\Psi^{''}-(1+\varepsilon)(\Psi^{'})^{2}\leq 0.
\end{equation}
Since $M^{n}$ is bounded and $C^{2}$, there exists a positive constant $a=a(\rho)>\sup\limits_{M^{n}}u$ such that
\begin{equation*}
-a\leq\vartheta < -\sup\limits_{M^{n}}u.
\end{equation*}
Let us take
\begin{equation*}
\Psi(\vartheta)=-\log(2a+\vartheta),
\end{equation*}
so we have \eqref{5.12} and
\begin{equation*}
\Psi^{'}\vartheta+1+c_{14}\left(\Psi^{''}-(1+\varepsilon)(\Psi^{'})^{2}\right)\geq\frac{1}{2}
\qquad \mathrm{for} ~~ \varepsilon\leq\frac{2a^{2}}{c_{14}}.
\end{equation*}
From \eqref{5.11}, together with
\begin{equation*}
F^{ij}h_{im}h_{jm}=F^{ii}h_{ii}^{2}\geq\frac{\zeta^{2}}{n}h_{11}^{2}\mathrm{tr}F^{ij},
\end{equation*}
which follows from the assumption \eqref{5.9} and the fact
$F^{nn}\geq\frac{1}{n}\mathrm{tr}F^{ij}$, at $X_{0}$, we have that
\begin{equation*}
\begin{split}
0&\geq -\frac{c_{12}\alpha}{\eta}-c_{9}
\left[\alpha+(1+\varepsilon^{-1})\alpha^{2}\right]\frac{\mathrm{tr}F^{ij}}{\eta^{2}}\\
&\quad +(\frac{\partial
f}{\partial\vartheta}\cdot\vartheta-f)h_{11}+\frac{\zeta^{2}}{2n}h_{11}^{2}\mathrm{tr}F^{ij}-c_{13},
\end{split}
\end{equation*}
which implies an upper bound
\begin{equation*}
\eta h_{11}\leq\frac{c_{15}}{\zeta}\qquad \mathrm{at}~~ X_{0}.
\end{equation*}
Since
\begin{equation*}
\mathrm{tr}F^{ij}=\sum_{i=1}^{n}\frac{\partial}{\partial\lambda_{i}}\left[\left(\frac{\sigma_{k}}{\sigma_{l}}\right)^{\frac{1}{k-l}}\right]
\geq\left(\frac{C^{k}_{n}}{C^{l}_{n}}\right)^{\frac{1}{k-l}}>0,
\end{equation*}
where $c_{15}$ is a positive constant depending on $c_{9}$, $c_{12}$, $c_{13}$, $\alpha$, $M^{n}$, $||\varphi||_{C^{0}(\overline{M^{n}})}$.

\textbf{Case 2}. We now assume that
\begin{equation}\label{5.13}
h_{nn}\geq-\zeta h_{11}.
\end{equation}
Since $h_{11}\geq h_{22}\geq\cdots\geq h_{nn}$, we have
\begin{equation*}
h_{ii}\geq-\zeta h_{11}\qquad \mathrm{for~~all}~~ i=1,\cdots,n.
\end{equation*}
For a positive constant, assume to be $4$, we divide
$\{1,\cdots,n\}$ into two parts as follows
\begin{equation*}
I=\{i:\mathcal{P}\leq 4\mathcal{P}^{11}\},\qquad J=\{j:\mathcal{P}^{jj}>4\mathcal{P}^{11}\},
\end{equation*}
where $\mathcal{P}^{ii}:=\frac{\partial\mathcal{P}}{\partial h_{ii}}=\mathcal{P}_{i}$ is evaluated at $\lambda(X_{0})$.
Then for each $i\in I$, by (\ref{5.4}), we have
\begin{equation*}
\begin{split}
\mathcal{P}_{i}\frac{|\nabla_{i}h_{11}|^{2}}{h_{11}^{2}}&=\mathcal{P}_{i}\left(\alpha\frac{\nabla_{i}\eta}{\eta}
+\Psi^{'}\nabla_{i}\vartheta\right)^{2}\\
&\leq(1+\varepsilon^{-1})\alpha^{2}\mathcal{P}_{i}\frac{|\nabla_{i}\eta|}{\eta^{2}}+(1+\varepsilon)(\Psi^{'})^{2}\mathcal{P}_{i}|\nabla_{i}\vartheta|^{2}
\end{split}
\end{equation*}
for any $\varepsilon>0$. For each $j\in J$, we have
\begin{equation*}
\begin{split}
\alpha\mathcal{P}_{i}\frac{|\nabla_{j}\eta|^{2}}{\eta^{2}}
&=\alpha^{-1}\mathcal{P}_{j}\left(\frac{\nabla_{j}h_{11}}{h_{11}}+\Psi^{'}\nabla_{j}\vartheta\right)^{2}\\
&\leq\frac{1+\varepsilon}{\alpha}(\Psi^{'})^{2}\mathcal{P}_{j}|\nabla_{j}\vartheta|^{2}
+\frac{1+\varepsilon^{-1}}{\alpha}\mathcal{P}_{j}\frac{|\nabla_{j}h_{11}|^{2}}{h_{11}^{2}}
\end{split}
\end{equation*}
for any $\varepsilon>0$. Consequently,
\begin{equation*}
\begin{split}
&\qquad
\alpha\sum_{i=1}^{n}\mathcal{P}_{i}\frac{|\nabla_{i}\eta|^{2}}{\eta^{2}}
+\sum_{i=1}^{n}\mathcal{P}_{i}\frac{|\nabla_{i}h_{11}|^{2}}{h_{11}^{2}}\\
&\leq
\left[\alpha+(1+\varepsilon^{-1})\alpha^{2}\right]\sum_{i\in I}\mathcal{P}_{i}\frac{|\nabla_{i}\eta|^{2}}{\eta^{2}}
+(1+\varepsilon)(\Psi^{'})^{2}\sum_{i\in I}\mathcal{P}_{i}|\nabla_{i}\vartheta|^{2}\\
&\quad
+\frac{1+\varepsilon}{\alpha}(\Psi^{'})^{2}\sum_{j\in J}\mathcal{P}_{j}|\nabla_{j}\vartheta|^{2}
+\left[1+(1+\varepsilon^{-1})\alpha^{-1}\right]\sum_{j\in J}\mathcal{P}_{j}\frac{|\nabla_{j}h_{11}|^{2}}{h_{11}^{2}}\\
&\leq 4n\left[\alpha+(1+\varepsilon^{-1})\alpha^{2}\right]\mathcal{P}_{1}\frac{|\nabla_{i}\eta|^{2}}{\eta^{2}}
+(1+\varepsilon)(1+\alpha^{-1})(\Psi^{'})^{2}\sum_{i=1}^{n}\mathcal{P}_{i}|\nabla_{i}\vartheta|^{2}\\
&\quad
+\left[1+(1+\varepsilon^{-1})\alpha^{-1}\right]\sum_{j\in J}\mathcal{P}_{j}\frac{|\nabla_{j}h_{11}|^{2}}{h_{11}^{2}}.
\end{split}
\end{equation*}
Using this estimate and \eqref{5.8}, the follows inequality
\begin{equation*}
\begin{split}
0&\geq-\frac{c_{6}\alpha}{\eta}
-4n\left[\alpha+(1+\varepsilon^{-1})\alpha^{2}\right]\mathcal{P}_{1}\frac{|\nabla_{i}\eta|^{2}}{\eta^{2}}
+\left[\Psi^{''}-(1+\varepsilon)(1+\alpha^{-1})(\Psi^{'})^{2}\right]\mathcal{P}_{i}|\nabla_{i}\vartheta|^{2}\\
&\quad
+\Psi^{'}\nabla_{l}f\langle X,X_{l}\rangle_{L}+(\frac{\partial f}{\partial\vartheta}\cdot\vartheta-f)h_{11}
+(\Psi^{'}\vartheta+1)F^{ij}h_{im}h_{jm}
+\frac{\partial f}{\partial\vartheta}\frac{\nabla_{l}h_{11}\langle X,X_{l}\rangle_{L}}{h_{11}}\\
&\quad
-\frac{1}{h_{11}}F^{ij,pq}\nabla_{1}h_{ij}\nabla_{1}h_{pq}
-\left[1+(1+\varepsilon^{-1})\alpha^{-1}\right]\sum_{j\in J}\mathcal{P}_{j}\frac{|\nabla_{j}h_{11}|^{2}}{h_{11}^{2}}-c_{13}
\end{split}
\end{equation*}
holds at $X_{0}$. Then as \textbf{Case 1}, we have that for an appropriate selection of $\Psi$,
\begin{equation}\label{5.14}
\begin{split}
0&\geq-\frac{c_{12}\alpha}{\eta}
-c_{16}(\alpha+\alpha^{2})\frac{\mathcal{P}_{1}}{\eta^{2}}+\frac{1}{2n}\mathcal{P}_{1}h_{11}^{2}
+(\frac{\partial f}{\partial\vartheta}\cdot\vartheta-f)h_{11}-c_{13}\\
&\quad
-\frac{1}{h_{11}}F^{ij,pq}\nabla_{1}h_{ij}\nabla_{1}h_{pq}
-[1+c_{17}\alpha^{-1}]\sum_{j\in J}\mathcal{P}_{j}\frac{|\nabla_{j}h_{11}|^{2}}{h_{11}^{2}},
\end{split}
\end{equation}
where $c_{16}>0$ depends on $n$, $\varepsilon^{-1}$, and $c_{17}=(1+\varepsilon^{-1})$.

We \textbf{claim} that
\begin{equation}\label{5.15}
\begin{split}
-\frac{1}{h_{11}}F^{ij,pq}\nabla_{1}h_{ij}\nabla_{1}h_{pq}-[1+c_{17}\alpha^{-1}]\sum_{j\in J}\mathcal{P}_{j}
\frac{|\nabla_{j}h_{11}|^{2}}{h_{11}^{2}}\geq 0.
\end{split}
\end{equation}
If \eqref{5.15} holds, then from \eqref{5.14} we have
\begin{equation*}
(\frac{\partial f}{\partial\vartheta}\cdot\vartheta-f)h_{11}+\frac{1}{2n}\mathcal{P}_{1}h_{11}^{2}\leq c_{18}(1+\frac{1}{\eta}+\frac{\mathcal{P}_{1}}{\eta^{2}}),
\end{equation*}
from which we again get a bound for $\eta h_{11}$ at $X_{0}$ due to
condition \eqref{f's conditions}, where $c_{18}>0$ depends on
$c_{12}$, $c_{13}$, $c_{16}$, $c_{17}$ and $\alpha$. The proof of
this\textbf{ claim} can be found at the end of the proof of
\cite[Theorem 1.3]{GLM}, so we omit it here. Then, the proof of
Theorem \ref{main1.2} is finished.
\end{proof}

\section{$C^{2}$ boundary estimates}
Throughout this section, we just assume that $M^{n}$ is strictly
convex. By N. M. Ivochkina, M. Lin, N. S. Trudinger \cite{nmi,mt1}
and Pierre Bayard \cite{Bayard}, we have the follow inequality:
$\exists\mathcal{B}_{0}=\mathcal{B}_{0}(n,k,l)$ such that in
$\Gamma_{k}$, $\forall i\in\{1,\cdots,k\}$,

\begin{equation}\label{sigma inequality}
\left( \frac{\sigma_{k}}{\sigma_{l}}
\right)_{\lambda_{i}}\cdot\lambda_{i}^{2}\leq\lambda_{i}\left(
\frac{\sigma_{k}}{\sigma_{l}} \right) + \mathcal{B}_{0}\sum_{i\neq
j}\left( \frac{\sigma_{k}}{\sigma_{l}}
\right)_{\lambda_{j}}\cdot\lambda_{j}^{2}.
\end{equation}

Let $x_{0}$ be a boundary-point, and $\{e_{1},\cdots,e_{n}\}$ be an
adapted basis such that at $x_{0}$,
$\sup\limits_{\overline{M^{n}}}\frac{|Du|}{u}=\sup\limits_{\overline{M^{n}}}|D\pi|\leq\rho<1$.

\begin{lemma}\label{lemma 6.1}
Let $\mathfrak{g}: \overline{M^{n}}\cap\overline{B}_{n}(x_{0})\times B(0,1)\to\mathbb{R}$, $(x,p)\mapsto\mathfrak{g}(x,p)$ be a function of class $C^{2}$, concave with respect to $p$, where $B_{r}(x_{0}):=\{x\in\mathbb{R}^{n+1}_{1}\mid |x-x_{0}|_{L}\leq r\}$, $B(0,1):=\{x\in\mathbb{R}^{n+1}_{1}\mid |x|_{L}<1\}$, and $\mathbb{W}=\mathfrak{g}(\cdot, Du)-\frac{\mathcal{B}}{2}\sum_{s=1}^{n-1}\left(u_{s}-u_{s}(x_{0})\right)^{2}$. If $D_{r,\rho}$ denotes the compact $\overline{M^{n}}\cap\overline{B}_{r}(x_{0})\times\overline{B}(0,\rho)$, for $\mathcal{B}=\mathcal{B}\left(n,k,l,\rho,\mathcal{B}_{0},\Vert\mathfrak{g}\Vert_{1,D_{r,\rho}}\right)$ sufficiently large, $\mathbb{W}$ satisfies on $M^{n}\cap B_{r}(x_{0})$ the inequality:
$$\sum_{i,j}\frac{\partial}{\partial q_{ij}}\left( \frac{\sigma_{k}}{\sigma_{l}}(u) \right)\mathbb{W}_{ij}\leq\mathcal{B}_{1}\left(1+|D\mathbb{W}|+\sum_{i,j}\frac{\partial}{\partial q_{ij}}\left( \frac{\sigma_{k}}{\sigma_{l}}(u) \right)\mathbb{W}_{i}\mathbb{W}_{j}+\frac{\sigma_{k-1}}{\sigma_{l}}(u)\right),$$
where $\mathcal{B}_{1} = \mathcal{B}_{1}\left(n,k,l,M^{n},\psi,\rho,\mathcal{B}_{0},\Vert\mathfrak{g}\Vert_{2,D_{r,\rho}}\right)$.

\end{lemma}

\begin{proof} Let us denote by $\tilde{\tau}_{\alpha}$, $\alpha=1,\cdots,n$, vectors of $\mathscr{H}^{n}(1)$ induced by the map $x\mapsto X:=(x,u(x))$,
 an orthonormal basis of principle vectors of $T_{X}\mathcal{G}$, and $\left( \tilde{\eta}_{s}^{\alpha} \right)$ such that, $\forall s\in\{1,\cdots,n\}$, $e_{s}=\sum_{\alpha=1}^{n}\tilde{\eta}^{\alpha}_{s}\tilde{\tau}_{\alpha}$.
 Define $\left(\tilde{\tau}_{\alpha}^{s} \right)$ such that, $\forall\alpha\in\{1,\cdots,n\}$, $\tilde{\tau}_{\alpha}=\sum_{s=1}^{n}\tilde{\tau}_{\alpha}^{s}e_{s}$. We thus have from the definition $\left(\tilde{\eta}_{s}^{\alpha} \right) = \left( \tilde{\tau}_{\alpha}^{s} \right)^{-1}$.

We will use the Greek letters and the Latin letters for derivatives
in the basis $\left\{\tilde{\tau}_{\alpha}, \alpha=1,\cdots,n
\right\}$ and $\left\{ e_{s}, s=1,\cdots,n \right\}$, respectively.
For instance, $u_{\alpha\beta}$ and $u_{s\alpha}$ will denote
respectively $D^{2}u(\tilde{\tau}_{\alpha}, \tilde{\tau}_{\beta})$
and $D^{2}u(e_{s}, \tilde{\tau}_{\alpha})$. In view of the choice of
the $\tilde{\tau}_{\alpha}$, the quantities
$\frac{1}{v}(u_{\alpha\alpha} + u - \frac{2}{u}u_{\alpha}^{2})$,
$\alpha = 1,\cdots,n$, are the principal curvatures of the graph
$\mathcal{G}$ of $u$. The inequality in Lemma \ref{lemma 6.1} may
then be written as
\begin{equation*}
\sum_{\alpha=1}^{n}\left( \frac{\sigma_{k}}{\sigma_{l}} \right)_{\lambda_{\alpha}}\mathbb{W}_{\alpha\alpha}
\leq
\mathcal{B}_{1}\left( 1+ |D\mathbb{W}| + \sum_{\alpha=1}^{n}\left( \frac{\sigma_{k}}{\sigma_{l}} \right)_{\lambda_{\alpha}}\mathbb{W}_{\alpha}^{2} + \frac{\sigma_{k-1}}{\sigma_{l}} \right).
\end{equation*}

For the first and second derivatives of $\mathbb{W}$, we have
$\forall\alpha\in\{1,\cdots,n\}$,
\begin{eqnarray*}
\mathbb{W}_{\alpha} = \mathfrak{g}_{\alpha} +
\sum_{t=1}^{n}\mathfrak{g}_{p_{t}}u_{t\alpha} -
\mathcal{B}\sum_{s=1}^{n-1}u_{s\alpha}\left(u_{s}-u_{s}(x_{0})\right),
 \end{eqnarray*}
  and
\begin{equation*}
\begin{aligned}
\mathbb{W}_{\alpha\alpha} &= \mathfrak{g}_{\alpha\alpha} + 2\sum_{t=1}^{n}\mathfrak{g}_{\alpha p_{t}}u_{t\alpha} + \sum_{s,t=1}^{n}\mathfrak{g}_{p_{t}p_{s}}u_{t\alpha}u_{s\alpha} \\
&\quad
+ \sum_{t=1}^{n}\mathfrak{g}_{p_{t}}u_{t\alpha\alpha} - \mathcal{B}\left(u_{s\alpha\alpha}(u_{s}-u_{s}(x_{0})) + u_{s\alpha}^{2} \right).
\end{aligned}
\end{equation*}
The following formula can represent the third derivatives of $u$
\begin{equation}\label{third derivative}
\begin{aligned}
&\quad\left( \frac{\sigma_{k}}{\sigma_{l}}(u) \right)_{,i}\\
& = \left( \frac{\sigma_{k}}{\sigma_{l}}(u)
\right)_{\lambda_{\alpha}}\cdot\lambda_{\alpha,i}\\
& =
\left( \frac{\sigma_{k}}{\sigma_{l}}(u) \right)_{\lambda_{\alpha}}\cdot\\
&\qquad
\left[ \frac{u_{\alpha}}{u}\cdot(\frac{u_{\alpha i}}{u} - \frac{u_{\alpha}u_{i}}{u^{2}})\cdot(u_{\alpha\alpha} + u - \frac{2}{u}u_{\alpha}^{2}) + (u_{\alpha\alpha i} + u_{i} - \frac{4u_{\alpha}u_{\alpha i}}{u} + \frac{2u_{\alpha}^{2}u_{i}}{u^{2}})\right]\\
&=
\psi_{i}.
\end{aligned}
\end{equation}
Since $u_{imm} = u_{mmi} - u_{i}$, by the use of $u_{s\alpha} = \tilde{\eta}_{s}^{\alpha}u_{\alpha\alpha}$, we get:
\begin{equation*}
\begin{aligned}
&\quad
\sum_{\alpha=1}^{n}\frac{\partial}{\partial\lambda_{\alpha}}\left( \frac{\sigma_{k}}{\sigma_{l}} \right)\mathbb{W}_{\alpha\alpha}\\
&=
\sum_{t=1}^{n}\mathfrak{g}_{p_{t}}\psi_{t} - \mathcal{B}\sum_{s=1}^{n-1}\psi_{s}\left(u_{s} - u_{s}(x_{0})\right) \\
&\quad
+ \sum_{\alpha=1}^{n}\left( \frac{\sigma_{k}}{\sigma_{l}} \right)_{\lambda_{\alpha}}\cdot
\left\{\mathfrak{g}_{\alpha\alpha} + 2\sum_{t=1}^{n}\tilde{\eta}_{t}^{\alpha}\mathfrak{g}_{\alpha p_{t}}u_{\alpha\alpha} + \sum_{s,t=1}^{n}\mathfrak{g}_{p_{t}p_{s}}\tilde{\eta}_{s}^{\alpha}\tilde{\eta}_{t}^{\alpha}u_{\alpha\alpha}^{2} \right.\\
&\quad
\left.- \mathcal{B}u_{\alpha\alpha}^{2}\sum_{s=1}^{n-1}|\tilde{\eta}_{s}^{\alpha}|^{2} +
\left(4\pi_{\alpha} - \frac{\pi_{\alpha}\lambda_{\alpha}}{u}\right)\mathbb{W}_{\alpha}+\left(\frac{\pi_{\alpha}\lambda_{\alpha}}{u} - 4\pi_{\alpha}\right)\mathfrak{g}_{\alpha}\right.\\
&\quad
\left. - \sum_{t=1}^{n}\mathfrak{g}_{p_{t}}\left[2u_{t}(1+\pi_{\alpha}^{2}) - \pi_{\alpha}^{2}\pi_{t}\lambda_{\alpha}\right] +
\mathcal{B}\sum_{s=1}^{n-1}\left[2u_{s}(1+\pi_{\alpha}^{2}) - \pi_{\alpha}^{2}\pi_{s}\lambda_{\alpha}\right]\left(u_{s}-u_{s}(x_{0})\right)
\right\},
\end{aligned}
\end{equation*}
where $\pi = \log u$. It's easy to estimate
$$\sum_{t=1}^{n}\mathfrak{g}_{p_{t}}\psi_{t} - \mathcal{B}\sum_{s=1}^{n-1}\psi_{s}\left(u_{s} - u_{s}(x_{0})\right) \leq b_{1}\left( 1+|D\mathbb{W}| \right),$$
where positive constant $b_{1}$ depends on $\psi$. Moreover, the term

\begin{equation*}
\begin{aligned}
\sum_{\alpha=1}^{n}\left( \frac{\sigma_{k}}{\sigma_{l}} \right)_{\lambda_{\alpha}}\left( \mathfrak{g}_{\alpha\alpha} + 2\sum_{t=1}^{n}\tilde{\eta}_{t}^{\alpha}\mathfrak{g}_{\alpha p_{t}}u_{\alpha\alpha} + \left(\frac{\pi_{\alpha}\lambda_{\alpha}}{u}- 4\pi_{\alpha}\right)\mathfrak{g}_{\alpha} - \sum_{t=1}^{n}\mathfrak{g}_{p_{t}}\left[2u_{t}(1+\pi_{\alpha}^{2}) - \pi_{\alpha}^{2}\pi_{t}\lambda_{\alpha}\right] \right)
\end{aligned}
\end{equation*}
is easily estimated by $b_{2}\left( \frac{\sigma_{k-1}}{\sigma_{l}}
+ \sum_{\alpha=1}^{n}\left( \frac{\sigma_{k}}{\sigma_{l}}
\right)_{\lambda_{\alpha}}|\lambda_{\alpha}| \right)$, positive
constant $b_{2}$ depends on $\Vert\mathfrak{g}\Vert_{2,D_{r,\rho}}$,
$\rho$, $M^{n}$, $k$, $l$. Recalling that $\mathfrak{g}$ is concave
w.r.t. $p$, for all $\alpha\in\{1,\cdots,n\}$,
$\sum_{s,t=1}^{n}\mathfrak{g}_{p_{t}p_{s}}\tilde{\eta}_{s}^{\alpha}\tilde{\eta}_{t}^{\alpha}\leq
0$, and then we thus finally get the estimate of
$\sum_{\alpha=1}^{n}\left( \frac{\sigma_{k}}{\sigma_{l}}
\right)_{\lambda_{\alpha}}\mathbb{W}_{\alpha\alpha}$:
\begin{equation}\label{finally estimate}
\begin{aligned}
&\quad
\sum_{\alpha=1}^{n}\left( \frac{\sigma_{k}}{\sigma_{l}} \right)_{\lambda_{\alpha}}\mathbb{W}_{\alpha\alpha}\\
&\leq
b_{3}\left( 1+|D\mathbb{W}| + \frac{\sigma_{k-1}}{\sigma_{l}} + \sum_{\alpha=1}^{n}\left( \frac{\sigma_{k}}{\sigma_{l}} \right)_{\lambda_{\alpha}}|\lambda_{\alpha}| \right)
+ \sum_{\alpha=1}^{n}\left( \frac{\sigma_{k}}{\sigma_{l}} \right)_{\lambda_{\alpha}}\left(-\mathcal{B}u_{\alpha\alpha}^{2}\sum_{s=1}^{n-1}|\tilde{\eta}_{s}^{\alpha}|^{2}\right. \\
&\qquad\qquad
\left. + \left( 4\pi_{\alpha} - \frac{\pi_{\alpha}\lambda_{\alpha}}{u} \right)\mathbb{W}_{\alpha} + \mathcal{B}\sum_{s=1}^{n-1}\left[ 2u_{s}(1+\pi_{\alpha}^{2}) - \pi_{\alpha}^{2}\pi_{s}\lambda_{\alpha} \right](u_{s}-u_{s}(x_{0}))\right)\\
&\leq
b_{3}\left( 1+|D\mathbb{W}| + \frac{\sigma_{k-1}}{\sigma_{l}} + \sum_{\alpha=1}^{n}\left( \frac{\sigma_{k}}{\sigma_{l}} \right)_{\lambda_{\alpha}}|\lambda_{\alpha}| \right)\\
&\quad
+\sum_{\alpha=1}^{n}\left( \frac{\sigma_{k}}{\sigma_{l}} \right)_{\lambda_{\alpha}}\left( -b_{4}\mathcal{B}\lambda_{\alpha}^{2}\sum_{s=1}^{n-1}|\tilde{\eta}_{s}^{\alpha}|^{2} + b_{5}|\mathbb{W}_{\alpha}| + b_{6}|\lambda_{\alpha}||\mathbb{W}_{\alpha}| \right),
\end{aligned}
\end{equation}
where the positive constant $b_{3}$ depends on $b_{1}$, $b_{2}$,
positive constants $b_{4}$, $b_{5}$, $b_{6}$ depend on
$\rho$, $M^{n}$. By \cite[Lemma 4.3]{Bayard}, we denote
$\delta_{\varepsilon}:=\delta_{\varepsilon}(\varepsilon,\rho,n)$,
where $\varepsilon\in(0,1)$. Let us consider two types of points:
either $\forall\alpha\in\{1,\cdots,n\}$,
$\sum_{s=1}^{n-1}|\tilde{\eta}_{s}^{\alpha}|^{2}\geq\delta_{\varepsilon}$,
or $\exists\alpha\in\{1,\cdots,n\}$ (for example $\alpha=1$) such
that
$\sum_{s=1}^{n-1}|\tilde{\eta}_{s}^{\alpha}|^{2}<\delta_{\varepsilon}$.

For the points of the first type, one has
\begin{equation*}
\begin{aligned}
&\quad
\sum_{\alpha=1}^{n}\left( \frac{\sigma_{k}}{\sigma_{l}} \right)_{\lambda_{\alpha}}\left( -b_{4}\mathcal{B}\lambda_{\alpha}^{2}\sum_{s=1}^{n-1}|\tilde{\eta}_{s}^{\alpha}|^{2} + b_{5}|\mathbb{W}_{\alpha}| + b_{6}|\lambda_{\alpha}||\mathbb{W}_{\alpha}|\right)\\
&\leq
\sum_{\alpha=1}^{n}\left( \frac{\sigma_{k}}{\sigma_{l}} \right)_{\lambda_{\alpha}}\left( b_{7}( 1+\frac{b_{8}}{\delta_{\varepsilon}} )\mathbb{W}_{\alpha}^{2} + b_{9}\delta_{\varepsilon}\lambda_{\alpha}^{2}  - b_{4}\mathcal{B}\delta_{\varepsilon}\lambda_{\alpha}^{2} \right),
\end{aligned}
\end{equation*}
where the positive constant $b_{7}$ depends on $b_{5}$, $b_{6}$, the
positive constant $b_{8}$ depends on $b_{6}$, $b_{7}$, and the positive
constant $b_{9}$ depends on $b_{6}$, $b_{7}$, $b_{8}$. Since
the term $b_{3}\sum_{\alpha=1}^{n}\left(
\frac{\sigma_{k}}{\sigma_{l}}
\right)_{\lambda_{\alpha}}|\lambda_{\alpha}|$ in \eqref{finally
estimate} can be estimated by
$$b_{3}\sum_{\alpha=1}^{n}\left( \frac{\sigma_{k}}{\sigma_{l}} \right)_{\lambda_{\alpha}}|\lambda_{\alpha}|
\leq \frac{b_{10}}{\delta_{\varepsilon}}\left(
\frac{\sigma_{k-1}}{\sigma_{l}} \right) +
\frac{\delta_{\varepsilon}}{2}\sum_{\alpha=1}^{n}\left(
\frac{\sigma_{k}}{\sigma_{l}}
\right)_{\lambda_{\alpha}}\cdot\lambda_{\alpha}^{2},$$ where the
positive constant $b_{10}$ depends on $b_{3}$, taking $\mathcal{B}$
large enough, we can get a bound on $\sum_{\alpha=1}^{n}\left(
\frac{\sigma_{k}}{\sigma_{l}}
\right)_{\lambda_{\alpha}}\mathbb{W}_{\alpha\alpha}$ of the expected
form.

For the points of the second type, let us consider two cases:

$\mathbf{First~case}$. $\lambda_{1}\leq 0$. The inequality
\eqref{sigma inequality} then becomes
\begin{equation}\label{sigma inequality 2}
\left( \frac{\sigma_{k}}{\sigma_{l}} \right)_{\lambda_{1}}\cdot\lambda_{1}^{2}\leq \mathcal{B}_{0}\sum_{\alpha=2}^{n}\left( \frac{\sigma_{k}}{\sigma_{l}} \right)_{\lambda_{\alpha}}\cdot\lambda_{\alpha}^{2}.
\end{equation}
From \cite[Lemma 4.3]{Bayard}, for $\alpha\geq2$,
$\exists\delta_{\varepsilon^{'}}:=\delta_{\varepsilon^{'}}(n,\rho,\varepsilon)$,
s.t.
$\sum_{s=1}^{n-1}|\tilde{\eta}_{s}^{\alpha}|^{2}\geq\delta_{\varepsilon^{'}}$.
Hence
$$\sum_{\alpha=2}^{n}\left( \frac{\sigma_{k}}{\sigma_{l}} \right)_{\lambda_{\alpha}}\cdot\lambda_{\alpha}^{2}\leq\frac{1}{\delta_{\varepsilon^{'}}}\left( \sum_{\alpha=2}^{n}( \frac{\sigma_{k}}{\sigma_{l}} )_{\lambda_{\alpha}}\cdot\lambda_{\alpha}^{2}\sum_{s=1}^{n-1}|\tilde{\eta}_{s}^{\alpha}|^{2} \right),$$
and by using inequality \eqref{sigma inequality 2}, it follows that
$$\sum_{\alpha=1}^{n}\left( \frac{\sigma_{k}}{\sigma_{l}} \right)_{\lambda_{\alpha}}\cdot\lambda_{\alpha}^{2}\leq\frac{\mathcal{B}_{0}+1}{\delta_{\varepsilon^{'}}}\left( \sum_{\alpha=1}^{n}( \frac{\sigma_{k}}{\sigma_{l}} )_{\lambda_{\alpha}}\cdot\lambda_{\alpha}^{2}\sum_{s=1}^{n-1}|\tilde{\eta}_{s}^{\alpha}|^{2} \right).$$
Thus, we can obtain
\begin{equation*}
\begin{aligned}
&\quad
\sum_{\alpha=1}^{n}\left( \frac{\sigma_{k}}{\sigma_{l}} \right)_{\lambda_{\alpha}}\left( -b_{4}\mathcal{B}\lambda_{\alpha}^{2}\sum_{s=1}^{n-1}|\tilde{\eta}_{s}^{\alpha}|^{2} + b_{5}|\mathbb{W}_{\alpha}| + b_{6}|\lambda_{\alpha}||\mathbb{W}_{\alpha}|  \right)\\
&\leq
\sum_{\alpha=1}^{n}\left( \frac{\sigma_{k}}{\sigma_{l}} \right)_{\lambda_{\alpha}}\left( b_{5}|\mathbb{W}_{\alpha}| + b_{6}|\lambda_{\alpha}||\mathbb{W}_{\alpha}|  - \frac{\mathcal{B}}{b_{11}}\lambda_{\alpha}^{2} \right)\\
&\leq
\sum_{\alpha=1}^{n}\left( \frac{\sigma_{k}}{\sigma_{l}} \right)_{\lambda_{\alpha}}\left( b_{7}( 1+\frac{b_{8}}{\delta_{\varepsilon}} )\mathbb{W}_{\alpha}^{2} + b_{9}\delta_{\varepsilon}\lambda_{\alpha}^{2} -\frac{\mathcal{B}}{b_{11}}\lambda_{\alpha}^{2} \right),
\end{aligned}
\end{equation*}
with
$b_{11}=\frac{\mathcal{B}_{0}+1}{b_{4}\delta_{\varepsilon^{'}}}$,
and
$$b_{3}\sum_{\alpha=1}^{n}\left( \frac{\sigma_{k}}{\sigma_{l}} \right)_{\lambda_{\alpha}}|\lambda_{\alpha}|
\leq b_{12}\left( \frac{\sigma_{k-1}}{\sigma_{l}} \right)+
\frac{1}{2b_{13}}\sum_{\alpha=1}^{n}\left(
\frac{\sigma_{k}}{\sigma_{l}}
\right)_{\lambda_{\alpha}}\cdot\lambda_{\alpha}^{2},$$ where
positive constants $b_{12}$, $b_{13}$ depend on $b_{3}$. Now one can
give the expected bound on $\sum_{\alpha=1}^{n}\left(
\frac{\sigma_{k}}{\sigma_{l}}
\right)_{\lambda_{\alpha}}\mathbb{W}_{\alpha\alpha}$ with
$\mathcal{B}$ large enough.

$\mathbf{Second~case}$. $\lambda_{1}>0$. Then one has
\begin{equation*}
\begin{aligned}
&\quad
\sum_{\alpha=1}^{n}\left( \frac{\sigma_{k}}{\sigma_{l}} \right)_{\lambda_{\alpha}}\left(b_{5}|\mathbb{W}_{\alpha}| + b_{6}|\lambda_{\alpha}||\mathbb{W}_{\alpha}|  \right)\\
&=
\left( \frac{\sigma_{k}}{\sigma_{l}} \right)_{\lambda_{1}}\left(b_{5}|\mathbb{W}_{1}| + b_{6}|\lambda_{1}||\mathbb{W}_{1}|  \right)\\
&\quad
+\sum_{\alpha=2}^{n}\left( \frac{\sigma_{k}}{\sigma_{l}} \right)_{\lambda_{\alpha}}\left(b_{5}|\mathbb{W}_{\alpha}| + b_{6}|\lambda_{\alpha}||\mathbb{W}_{\alpha}|  \right)\\
&\leq
b_{14}\left( \frac{\sigma_{k}}{\sigma_{l}} \right)_{\lambda_{1}}\cdot\lambda_{1}\mathbb{W}_{1} +
b_{15}\sum_{\alpha=1}^{n}\left( \frac{\sigma_{k}}{\sigma_{l}} \right)_{\lambda_{\alpha}}\mathbb{W}_{\alpha}^{2} + \sum_{\alpha=2}^{n}\left( \frac{\sigma_{k}}{\sigma_{l}} \right)_{\lambda_{\alpha}}\cdot\lambda_{\alpha}^{2},
\end{aligned}
\end{equation*}
where the positive constants $b_{14}$, $b_{15}$ depend on $b_{5}$, $b_{6}$. Let us bound $Q=\left( \frac{\sigma_{k}}{\sigma_{l}} \right)_{\lambda_{1}}\lambda_{1}\mathbb{W}_{1}$. $Q=\mathbb{W}_{1}\left( \psi-\widetilde{\Gamma}(k,l,1) \right)$ with $\psi = \left( \frac{\sigma_{k}}{\sigma_{l}} \right)_{\lambda_{\alpha}}\cdot\lambda_{\alpha} + \widetilde{\Gamma}(k,l,\alpha)$, where $\alpha=1,\cdots,n$, $\widetilde{\Gamma}(k,l,\alpha) = \frac{\sigma_{l}\sigma_{k}(\lambda|\alpha) - \sigma_{k}\sigma_{l}(\lambda|\alpha)}{\sigma_{l}^{2}} + \frac{\sigma_{k}}{\sigma_{l}}$, and $\left( \frac{\sigma_{k}}{\sigma_{l}} \right)_{\lambda_{1}}\lambda_{1}$ is positive. We must again consider two cases:

If $\widetilde{\Gamma}(k,l,1)\geq 0$,
$$|Q|\leq|\mathbb{W}_{1}|\left(\psi - \widetilde{\Gamma}(k,l,1)\right)\leq|\mathbb{W}_{1}|\psi\leq b_{16}|D\mathbb{W}|,$$
where the positive constant $b_{16}$ depends on $\psi$.

If $\widetilde{\Gamma}(k,l,1)<0$,
\begin{equation}\label{f1}
|Q|\leq -|\mathbb{W}_{1}|\widetilde{\Gamma}(k,l,1) + |\mathbb{W}_{1}|\psi \leq -|\mathbb{W}_{1}|\widetilde{\Gamma}(k,l,1) + b_{16}|D\mathbb{W}|.
\end{equation}
Since
\begin{equation*}
\mathbb{W}_{1} = \mathfrak{g}_{1} + \sum_{t=1}^{n}\mathfrak{g}_{p_{t}}u_{11}\tilde{\eta}_{t}^{1} - \mathcal{B}\sum_{s=1}^{n-1}u_{11}\tilde{\eta}_{s}^{1}\left( u_{s} - u_{s}(x_{0}) \right).
\end{equation*}
If $\left|\sum_{t=1}^{n}\mathfrak{g}_{p_{t}}\tilde{\eta}_{t}^{1} -
\mathcal{B}\sum_{s=1}^{n}\tilde{\eta}_{s}^{1}\left(
u_{s}-u_{s}(x_{0}) \right)\right|$ is less than $\tilde{\beta}$ and
$|\mathfrak{g}_{1}|$ is less than $b_{17}$, then
 \begin{eqnarray*}
|\mathbb{W}_{1}|\leq\tilde{\beta}\lambda_{1} + b_{17}.
 \end{eqnarray*}
Besides, $\forall s\in\{1,\cdots,n-1\}$,
$|\tilde{\eta}_{s}^{1}|\leq(\delta_{\varepsilon})^{1/2}$. Choosing
$\varepsilon=\varepsilon\left(n,\rho,\mathcal{B}\right)$
sufficiently small such that
\begin{equation}\label{f2}
\mathcal{B}(\delta_{\varepsilon})^{1/2} \leq 1,
\end{equation}
and then, with such a choice of $\varepsilon$ we get: $\forall
s\in\{1,\cdots,n-1\}$, $|\mathcal{B}\tilde{\eta}_{s}^{1}|\leq 1$.
Let us then take $\tilde{\beta}$ as
\begin{equation}\label{f3}
\tilde{\beta} = \sup\limits_{D_{r,\rho}}\left( \sum_{t=1}^{n}|\mathfrak{g}_{p_{t}}| + 2(n-1) \right).
\end{equation}
Hence,
\begin{equation}\label{f4}
\begin{aligned}
-|\mathbb{W}_{1}|\widetilde{\Gamma}(k,l,1) &\leq
-\widetilde{\Gamma}(k,l,1)\tilde{\beta}\lambda_{1} - b_{17}\widetilde{\Gamma}(k,l,1)\\
&\leq
-\widetilde{\Gamma}(k,l,1)\tilde{\beta}\lambda_{1} + b_{17}\left( ( \frac{\sigma_{k}}{\sigma_{l}} )_{\lambda_{1}}\cdot\lambda_{1} - \psi \right).
\end{aligned}
\end{equation}
Let us find a bound on $-\widetilde{\Gamma}(k,l,1)\cdot\lambda_{1}$:
\begin{equation*}
\begin{aligned}
\sum_{\alpha=2}^{n}\left( \frac{\sigma_{k}}{\sigma_{l}} \right)_{\lambda_{\alpha}}\cdot\lambda_{\alpha}^{2}
&=
\sum_{\alpha=2}^{n}\left( \psi - \widetilde{\Gamma}(k,l,\alpha) \right)\lambda_{\alpha}\\
&=
\psi\sigma_{1}(\lambda|1) - \sum_{\alpha=2}^{n}\widetilde{\Gamma}(k,l,\alpha)\lambda_{\alpha}\\
&=
\psi\sigma_{1}(\lambda|1) - \sum_{\alpha=1}^{n}\widetilde{\Gamma}(k,l,\alpha)\lambda_{\alpha} + \widetilde{\Gamma}(k,l,1)\lambda_{1}\\
&=
\psi\sigma_{1}(\lambda|1) + (k-l)\frac{\sigma_{k}}{\sigma_{l}} - \frac{\sigma_{k}}{\sigma_{l}}\sigma_{1} + \widetilde{\Gamma}(k,l,1)\lambda_{1}.
\end{aligned}
\end{equation*}
So, we have

\begin{equation}\label{f5}
\begin{aligned}
-\widetilde{\Gamma}(k,l,1)\lambda_{1} &=
\psi\sigma_{1}(\lambda|1) + (k-l)\frac{\sigma_{k}}{\sigma_{l}} - \frac{\sigma_{k}}{\sigma_{l}}\sigma_{1} - \sum_{\alpha=2}^{n}\left( \frac{\sigma_{k}}{\sigma_{l}} \right)_{\lambda_{\alpha}}\cdot\lambda_{\alpha}^{2}\\
&\leq
-\psi\lambda_{1} + (k-l)\frac{\sigma_{k}}{\sigma_{l}} + \sum_{\alpha=2}^{n}\left( \frac{\sigma_{k}}{\sigma_{l}} \right)_{\lambda_{\alpha}}\cdot\lambda_{\alpha}^{2}\\
&\leq
(k-l)\frac{\sigma_{k}}{\sigma_{l}} + \sum_{\alpha=2}^{n}\left( \frac{\sigma_{k}}{\sigma_{l}} \right)_{\lambda_{\alpha}}\cdot\lambda_{\alpha}^{2}.
\end{aligned}
\end{equation}
Last, we estimate $\left( \frac{\sigma_{k}}{\sigma_{l}} \right)_{\lambda_{1}}\cdot\lambda_{1}$ as follows:

\begin{equation}\label{f6}
\begin{aligned}
\left( \frac{\sigma_{k}}{\sigma_{l}} \right)_{\lambda_{1}}\cdot\lambda_{1} &=
(k-l)\frac{\sigma_{k}}{\sigma_{l}} - \sum_{\alpha=2}^{n}\left( \frac{\sigma_{k}}{\sigma_{l}} \right)_{\lambda_{\alpha}}\cdot\lambda_{\alpha}\\
&\leq
(k-l)\frac{\sigma_{k}}{\sigma_{l}} + \frac{1}{4\gamma}\sum_{\alpha=2}^{n}\left( \frac{\sigma_{k}}{\sigma_{l}} \right)_{\lambda_{\alpha}}\cdot\lambda_{\alpha} + \gamma\sum_{\alpha=2}^{n}\left( \frac{\sigma_{k}}{\sigma_{l}} \right)_{\lambda_{\alpha}}\cdot\lambda_{\alpha}^{2}\\
&\leq
(k-l)\frac{\sigma_{k}}{\sigma_{l}} + \frac{1}{4\gamma}\frac{(n-k+1)\sigma_{k-1}\sigma_{l} - (n-l+1)\sigma_{k}\sigma_{l-1}}{\sigma_{l}^{2}} \\
&\quad
+ \gamma\sum_{\alpha=2}^{n}\left( \frac{\sigma_{k}}{\sigma_{l}} \right)_{\lambda_{\alpha}}\cdot\lambda_{\alpha}^{2}\\
&\leq
(k-l)\frac{\sigma_{k}}{\sigma_{l}} + \frac{n-k+1}{4\gamma}\frac{\sigma_{k-1}}{\sigma_{l}} + \gamma\sum_{\alpha=2}^{n}\left( \frac{\sigma_{k}}{\sigma_{l}} \right)_{\lambda_{\alpha}}\cdot\lambda_{\alpha}^{2}\\
&\leq
b_{18}\left( 1+\frac{\sigma_{k-1}}{\sigma_{l}} \right) + \gamma\sum_{\alpha=2}^{n}\left( \frac{\sigma_{k}}{\sigma_{l}} \right)_{\lambda_{\alpha}}\cdot\lambda_{\alpha}^{2},
\end{aligned}
\end{equation}
where the positive constant $b_{18}$ depends on $\psi$, $k$, $l$, $\gamma$, and $\gamma$ is to be specified. With $\gamma$ sufficiently small, inequalities \eqref{f1}, \eqref{f4}, \eqref{f5} and \eqref{f6} then give
$$|Q| \leq b_{19}\left( 1+|D\mathbb{W}| + \frac{\sigma_{k-1}}{\sigma_{l}} \right) + b_{20}\sum_{\alpha=2}^{n}\left( \frac{\sigma_{k}}{\sigma_{l}} \right)_{\lambda_{\alpha}}\cdot\lambda_{\alpha}^{2},$$
where the positive constant $b_{19}$ depends on $\tilde{\beta}$, $k$, $l$, $\psi$, $b_{16}$, $b_{17}$, $b_{18}$, the positive constant $b_{20}$ depends on $\tilde{\beta}$, $b_{17}$. Recalling the inequality

\begin{equation*}
\begin{aligned}
&\quad
\sum_{\alpha=1}^{n}\left( \frac{\sigma_{k}}{\sigma_{l}} \right)_{\lambda_{\alpha}}\cdot\mathbb{W}_{\alpha\alpha}\\
&\leq
b_{3}\left( 1 + |D\mathbb{W}| + \frac{\sigma_{k-1}}{\sigma_{l}} + \sum_{\alpha=1}^{n}( \frac{\sigma_{k}}{\sigma_{l}} )_{\lambda_{\alpha}}|\lambda_{\alpha}| \right) - b_{4}\mathcal{B}\delta_{\varepsilon^{'}}\sum_{\alpha=2}^{n}\left( \frac{\sigma_{k}}{\sigma_{l}} \right)_{\lambda_{\alpha}}\cdot\lambda_{\alpha}^{2}\\
&\quad +\sum_{\alpha=1}^{n}\left( \frac{\sigma_{k}}{\sigma_{l}}
\right)_{\lambda_{\alpha}}\left( b_{5}|\mathbb{W}_{\alpha}| +
b_{6}|\lambda_{\alpha}||\mathbb{W}_{\alpha}|  \right),
\end{aligned}
\end{equation*}
with estimates \eqref{f6},
\begin{equation*}
b_{3}\sum_{\alpha=1}^{n}\left( \frac{\sigma_{k}}{\sigma_{l}}
\right)_{\lambda_{\alpha}}\cdot|\lambda_{\alpha}|\leq b_{21}\left(
1+\frac{\sigma_{k-1}}{\sigma_{l}} \right) +
b_{22}\sum_{\alpha=2}^{n}\left( \frac{\sigma_{k}}{\sigma_{l}}
\right)_{\lambda_{\alpha}}\cdot\lambda_{\alpha}^{2}
\end{equation*}
and
\begin{equation*}
\begin{aligned}
&\quad
\sum_{\alpha=1}^{n}\left( \frac{\sigma_{k}}{\sigma_{l}} \right)_{\lambda_{\alpha}}\left(b_{5}|\mathbb{W}_{\alpha}| + b_{6}|\lambda_{\alpha}||\mathbb{W}_{\alpha}|  \right)\\
&\leq
b_{14}|Q| + b_{15}\sum_{\alpha=1}^{n}\left( \frac{\sigma_{k}}{\sigma_{l}} \right)_{\lambda_{\alpha}}\mathbb{W}_{\alpha}^{2} + \sum_{\alpha=2}^{n}\left( \frac{\sigma_{k}}{\sigma_{l}} \right)_{\lambda_{\alpha}}\cdot\lambda_{\alpha}^{2}\\
&\leq
b_{23}\left( 1+|D\mathbb{W}| + \frac{\sigma_{k-1}}{\sigma_{l}} \right) + b_{24}\sum_{\alpha=2}^{n}\left( \frac{\sigma_{k}}{\sigma_{l}} \right)_{\lambda_{\alpha}}\cdot\lambda_{\alpha}^{2} + b_{15}\sum_{\alpha=1}^{n}\left( \frac{\sigma_{k}}{\sigma_{l}} \right)_{\lambda_{\alpha}}\mathbb{W}_{\alpha}^{2},
\end{aligned}
\end{equation*}
where the positive constant $b_{21}$ depends on $b_{3}$, $b_{18}$,
the positive constant $b_{22}$ depends on $b_{3}$, $\gamma$, the
positive constant $b_{23}$ depends on $b_{14}$, $b_{19}$, and the
positive constant $b_{24}$ depends on $b_{20}$. Let $\mathcal{B} =
\mathcal{B}\left( n,k,l,\rho,\Vert\mathfrak{g}\Vert_{1,D_{r,\rho}},
\mathcal{B}_{0} \right)$ be large enough, independent of
$\varepsilon$, and compatible with \eqref{f2}. Then we have
\begin{equation*}
\begin{aligned}
&\quad
\sum_{\alpha=1}^{n}\left( \frac{\sigma_{k}}{\sigma_{l}} \right)_{\lambda_{\alpha}}\cdot\mathbb{W}_{\alpha\alpha}\\
&\leq
b_{25}\left( 1 + |D\mathbb{W}| + \frac{\sigma_{k-1}}{\sigma_{l}} +\sum_{\alpha=1}^{n}\left( \frac{\sigma_{k}}{\sigma_{l}} \right)_{\lambda_{\alpha}}\mathbb{W}_{\alpha}^{2} \right) \\
&\quad
+ b_{26}\sum_{\alpha=2}^{n}\left( \frac{\sigma_{k}}{\sigma_{l}} \right)_{\lambda_{\alpha}}\cdot\lambda_{\alpha}^{2} - b_{4}\mathcal{B}\delta_{\varepsilon^{'}}\sum_{\alpha=2}^{n}\left( \frac{\sigma_{k}}{\sigma_{l}} \right)_{\lambda_{\alpha}}\cdot\lambda_{\alpha}^{2},
\end{aligned}
\end{equation*}
where the positive constant $b_{25}$ depends on $b_{3}$, $b_{15}$,
$b_{21}$, $b_{23}$, and the positive constant $b_{26}$ depends on
$b_{22}$, $b_{24}$. Now we can give the expected bound on
$\sum_{\alpha=1}^{n}\left( \frac{\sigma_{k}}{\sigma_{l}}
\right)_{\lambda_{\alpha}}\mathbb{W}_{\alpha\alpha}$ with
$\mathcal{B}$ large enough. This completes the proof of Lemma
\ref{lemma 6.1}.
\end{proof}

Setting $\widetilde{\mathbb{W}} = \exp\left(
-\mathcal{B}_{1}\mathfrak{g}(x_{0}, Du(x_{0})) \right) - \exp\left(
-\mathcal{B}_{1}\mathbb{W} \right) - b|x-x_{0}|^{2}$, from Lemma
\ref{lemma 6.1} and an appropriate $b$ as in \cite{nmi}, we can get
the following crucial inequality.

\begin{lemma}\label{lemma 6.2}
For $b = b\left(
n,\mathcal{B}_{1},\Vert\mathfrak{g}\Vert_{0,D_{r,\rho}},\mathcal{B}
\right)$ sufficiently large, $\widetilde{\mathbb{W}}$ satisfies
$$\sum_{i,j}\frac{\partial}{\partial q_{ij}}\left( \frac{\sigma_{k}}{\sigma_{l}}(u)\right)\widetilde{\mathbb{W}}_{ij} \leq \mathcal{B}_{2}\left( 1 + |D\widetilde{\mathbb{W}}| \right)$$
on $M^{n}\cap B_{r}(x_{0})$, where $\mathcal{B}_{2} =
\mathcal{B}_{2}\left( \mathcal{B}_{1}, b, r,
\Vert\mathfrak{g}\Vert_{0,D_{r,\rho}}, \mathcal{B} \right)$.
\end{lemma}

\begin{lemma}\label{lemma 6.3}
Let $\tilde{\psi}\in C^{2}\left( \partial M^{n}\cap
\bar{B}_{r}(x_{0}) \right)$ and $a_{0}\in\mathbb{R}$. The function
$\tilde{v} = -a_{0}|x-x_{0}|^{2} - \tilde{h}(d) +
\tilde{\psi}(x^{'})$, with $\tilde{h}(d) = \mathcal{B}_{3}\left( 1 -
e^{-\mathcal{B}_{4}d} \right)$, satisfies: for any positive function
$p$, $Dp\in\mathbb{R}^{n}$, s.t. $\frac{|Dp|}{p}\leq\rho<1$,
$\forall i\in\{1,\cdots,k\}$,
$\mathcal{F}_{i}(Dp,D^{2}\tilde{v})>0$, and
\begin{equation}\label{f7}
\sum_{i,j}\frac{\partial}{\partial
q_{ij}}\left(\frac{\mathcal{F}_{k}}{\mathcal{F}_{l}}(u)\right)\tilde{v}_{ij}\geq
\mathcal{B}_{2}(1+|D\tilde{v}|)  \qquad on~~M^{n}\cap B_{r}(x_{0})
\end{equation}
for suitable parameters $\mathcal{B}_{3}$, $\mathcal{B}_{4}$, depending only on $n$, $k$, $l$, $M^{n}$, $\psi$, $\rho$, $a_{0}$, $\Vert\tilde{\psi}\Vert_{2,\partial M^{n}\cap \bar{B}_{r}(x_{0})}$ and $\mathcal{B}_{2}$.
\end{lemma}

Here $d$ denotes the distance function to the boundary of $M^{n}$,
$x=(x^{'},x_{n})$ in a given adapted basis at the boundary-point
$x_{0}$, and $\mathcal{F}_{k}(Dp,q) = F_{k}\left( g^{ij}(Dp)q
\right)$, $g^{ij}(Dp) = \frac{1}{p^{2}}\left( \sigma^{ij} +
\frac{p^{i}p^{j}}{p^{2}-|Dp|^{2}} \right)$, $F_{k}(q)$ denotes the
$k^{th}$ symmetric function of the eigenvalues of $q$. Particularly,
$\mathcal{F}_{k}(u) = \sigma_{k}(u)$.

To prove Lemma \ref{lemma 6.3}, we will need the following fact.

\begin{lemma}\label{lemma 6.4}
If $q$ is a symmetric non-negative matrix, then for any positive
function $p$, $Dp\in\mathbb{R}^{n}$, s.t.
$\frac{|Dp|}{p}\leq\rho<1$, $2\leq k\leq n$, $0\leq l\leq k-2$,
 $$\frac{\mathcal{F}_{k}}{\mathcal{F}_{l}}(Dp,q)\geq(1-\rho^{2})\left( \frac{1}{p^{2}} \right)^{k-l}\frac{F_{k}}{F_{l}}(q).$$
\end{lemma}

\begin{proof} Set $2\leq k\leq n$, $0 \leq l\leq k-2$, $\forall (Dp,q)\in B(0,1)\times S_{n}(\mathbb{R})$, where $S_{n}(\mathbb{R})$ denotes $n\times n$ order real symmetric matrix,
\begin{equation}\label{f8}
\mathcal{F}_{k}(Dp,q) = F_{k}\left( g^{ij}(Dp)q\right).
\end{equation}
The expression \eqref{f8}, independent of the orthonormal basis of
$\mathbb{R}^{n}$, is chosen to express $p$ and $q$. Take an
orthonormal basis $\{e_{1}, \cdots, e_{n}\}$ with $e_{1}$ directed
along $Dp$. Then
\begin{equation*}
g^{ij}(Dp)q =
\begin{pmatrix}
\frac{1}{p^{2} - |Dp|^{2}} &  &  &  \\
 & 1/p^{2} &  &  \\
 &  & \ddots &  \\
 &  &  & 1/p^{2}
\end{pmatrix}q
=
\begin{pmatrix}
\frac{1}{p^{2} - |Dp|^{2}}q_{11} & \frac{1}{p^{2} - |Dp|^{2}}q_{12} & \cdots &  \frac{1}{p^{2} - |Dp|^{2}}q_{1n}\\
1/p^{2}\cdot q_{21} & 1/p^{2}\cdot q_{22} & \cdots & 1/p^{2}\cdot q_{2n}\\
\vdots & \vdots &  & \vdots\\
1/p^{2}\cdot q_{n1} & 1/p^{2}\cdot q_{n2} & \cdots & 1/p^{2}\cdot q_{nn}
\end{pmatrix}.
\end{equation*}
Thus
\begin{equation*}
\begin{aligned}
F_{k}\left( g^{ij}(Dp)q \right) & =  \frac{1}{p^{2} - |Dp|^{2}}\times\left(\frac{1}{p^{2}}\right)^{k-1}\sum\limits_{1<i_{1}<\cdots<i_{k}\leq n}
\left|\begin{matrix}
q_{11} & q_{1i_{2}} & \cdots & q_{1i_{k}}\\
q_{i_{2}1} & q_{i_{2}i_{2}} & \cdots & q_{i_{2}}q_{i_{k}}\\
\vdots & \vdots &  & \vdots\\
q_{i_{k}1} & q_{i_{k}i_{2}} & \cdots & q_{i_{k}i_{k}}
\end{matrix}\right|\\
&\quad +\left( \frac{1}{p^{2}}
\right)^{k}\sum\limits_{\substack{I=(i_{1}, \cdots,
i_{k})\\1<i_{1}<\cdots<i_{k}\leq n}}q_{II} \geq \left(
\frac{1}{p^{2}}  \right)^{k}F_{k}(q),
\end{aligned}
\end{equation*}
\begin{equation*}
\begin{aligned}
F_{l}\left( g^{ij}(Dp)q \right) & =  \frac{1}{p^{2} - |Dp|^{2}}\times\left(\frac{1}{p^{2}}\right)^{l-1}\sum\limits_{1<i_{1}<\cdots<i_{l}\leq n}
\left|\begin{matrix}
q_{11} & q_{1i_{2}} & \cdots & q_{1i_{l}}\\
q_{i_{2}1} & q_{i_{2}i_{2}} & \cdots & q_{i_{2}}q_{i_{l}}\\
\vdots & \vdots &  & \vdots\\
q_{i_{l}1} & q_{i_{l}i_{2}} & \cdots & q_{i_{l}i_{l}}
\end{matrix}\right|\\
&\quad
+\left( \frac{1}{p^{2}}  \right)^{l}\sum\limits_{\substack{J=(i_{1}, \cdots, i_{l})\\1<i_{1}<\cdots<i_{l}\leq n}}q_{JJ} \leq \frac{1}{1-\rho^{2}}\left( \frac{1}{p^{2}}  \right)^{l}F_{l}(q).
\end{aligned}
\end{equation*}
Using $\frac{|Dp|}{p}\leq\rho<1$, since the determinants in the
first sum are non-negative when $q$ is non-negative, Lemma
\ref{lemma 6.4} follows.
\end{proof}

\begin{proof} [Proof of Lemma \ref{lemma 6.3}] We first show that for any positive function
$p$, $Dp\in\mathbb{R}^{n}$, s.t. $\frac{|Dp|}{p}\leq\rho<1$, $0\leq
l\leq k-2$, $2\leq k\leq n$,
\begin{equation}\label{f9}
\frac{\mathcal{F}_{k}}{\mathcal{F}_{l}}(Dp, D^{2}\tilde{v}) \geq \mathcal{B}_{5}(1+|D\tilde{v}|)^{k-l},
\end{equation}
where $\mathcal{B}_{5}$ is as large as desired, and if
$\mathcal{B}_{3}$, $\mathcal{B}_{4}$ are suitable parameters. Let
$\{e_{1},\cdots,e_{n}\}$ be an orthonormal basis of
$\mathscr{H}^{n}(1)$, $q\in S_{n}(\mathbb{R})$. Denote by $Dp^{'}$
the component of $Dp$ on $\{e_{1},\cdots,e_{n-1}\}$, by $q^{'}\in
S_{n-1}(\mathbb{R})$ the restriction of $q$ on
$\{e_{1},\cdots,e_{n-1}\}$, and by $q_{(n,n)}$ the $n\times n$
matrix deduced from $q$ by setting its $(n,n)$ coefficient to $0$.
It's easy to show that
\begin{equation}\label{f10}
\mathcal{F}_{k}(Dp,q) = \frac{1}{p^{2}}
\frac{p^{2}-|Dp^{'}|^{2}}{p^{2}-|Dp|^{2}}q_{nn}\mathcal{F}_{k-1}\left(Dp^{'},q^{'}\right)
+ O\left( |q_{(n,n)}|^{k} \right),
\end{equation}
\begin{equation}\label{f11}
\mathcal{F}_{l}(Dp,q) = \frac{1}{p^{2}}
\frac{p^{2}-|Dp^{'}|^{2}}{p^{2}-|Dp|^{2}}q_{nn}\mathcal{F}_{l-1}\left(Dp^{'},q^{'}\right)
+ O\left( |q_{(n,n)}|^{l} \right),
\end{equation}
where $O\left( |q_{(n,n)}|^{k} \right)$, $O\left( |q_{(n,n)}|^{l}
\right)$ denote  quantities estimated by $C|q_{(n,n)}|^{k}$,
$C|q_{(n,n)}|^{l}$ with $C$ depending on $Dp$, $p$. So, we have
\begin{equation}\label{f12}
\frac{\mathcal{F}_{k}}{\mathcal{F}_{l}}(Dp,q) = \frac{\left( p^{2} - |Dp^{'}|^{2} \right)q_{nn}\mathcal{F}_{k-1}\left( Dp^{'}, q^{'} \right) + O\left( |q_{(n,n)}|^{k} \right)}{\left( p^{2} - |Dp^{'}|^{2} \right)q_{nn}\mathcal{F}_{l-1}\left( Dp^{'},q^{'} \right) + O\left( |q_{(n,n)}|^{l} \right)}.
\end{equation}
Let $x\in M^{n}\cap B_{r}(x_{0})$, and $\{e_{1},\cdots,e_{n}\}$ be
an adapted basis at the boundary-point $y$ minimizing the distance
between $x$ and $\partial M^{n}$. We get
\begin{equation}\label{f13}
\frac{\mathcal{F}_{k}}{\mathcal{F}_{l}}(Dp,D^{2}\tilde{v}) = \frac{\left( p^{2} - |Dp^{'}|^{2} \right)\tilde{v}_{nn}\mathcal{F}_{k-1}\left( Dp^{'}, D^{2}\tilde{v}^{'} \right) + O\left( |D^{2}\tilde{v}_{(n,n)}|^{k} \right)}{\left( p^{2} - |Dp^{'}|^{2} \right)\tilde{v}_{nn}\mathcal{F}_{l-1}\left( Dp^{'},D^{2}\tilde{v}^{'} \right) + O\left( |D^{2}\tilde{v}_{(n,n)}|^{l} \right)},
\end{equation}
where

\begin{equation*}
\begin{aligned}
D^{2}\tilde{v} &=
\tilde{h}^{'}\mathrm{diag}\left( -\frac{2a_{0}}{\tilde{h}^{'}}+\frac{\kappa_{1}}{1-\kappa_{1}d}, \cdots, -\frac{2a_{0}}{\tilde{h}^{'}} + \frac{\kappa_{n-1}}{1-\kappa_{n-1}d}, -\frac{2a_{0}}{\tilde{h}^{'}} - \frac{\tilde{h}^{''}}{\tilde{h}^{'}} \right)\\
&\quad
+
\begin{pmatrix}
\left( \frac{\partial^{2}\tilde{\psi}}{\partial x_{i}\partial x_{j}} \right)_{1\leq i,j\leq n-1} & 0\\
0 & 0
\end{pmatrix}.
\end{aligned}
\end{equation*}
Let us bound $\frac{\mathcal{F}_{k-1}}{\mathcal{F}_{l-1}}\left(
Dp^{'}, D^{2}\tilde{v}^{'} \right)$ from below: from the definition
of $\mathcal{F}_{k-1}$, $\mathcal{F}_{l-1}$, one knows that
$\left(\frac{1}{\tilde{h}^{'}}
\right)^{k-l}\frac{\mathcal{F}_{k-1}}{\mathcal{F}_{l-1}}(Dp^{'},
D^{2}\tilde{v}^{'} )$ tends to
\begin{equation}\label{f14}
\frac{F_{k-1}}{F_{l-1}}\left( g^{ij}(
Dp^{'})\mathrm{diag}(\kappa_{1},\cdots,\kappa_{n-1}) \right),
\end{equation}
where $d$ tends to $0$, $\tilde{h}^{'}$ tends to $+\infty$,
$\kappa_{1}, \cdots, \kappa_{n-1}$ denote the principal curvatures
of $\partial M^{n}$ associated with $\{e_{1}, \cdots, e_{n-1}\}$. By
Lemma \ref{lemma 6.4}, \eqref{f14} can be estimated from below.
Taking $r$ small and $\tilde{h}^{'}$ large, we thus obtain the
estimate

\begin{equation*}
\frac{\mathcal{F}_{k-1}}{\mathcal{F}_{l-1}}\left( Dp^{'}, D^{2}\tilde{v}^{'} \right) \geq \delta\left( \tilde{h}^{'} \right)^{k-l},
\end{equation*}
where $\delta$ is a positive constant under control. This allows the estimate of $\frac{\mathcal{F}_{k}}{\mathcal{F}_{l}}\left( Dp,D^{2}\tilde{v} \right)$: from \eqref{f13}, we have

\begin{equation*}
\frac{\mathcal{F}_{k}}{\mathcal{F}_{l}}\left( Dp, D^{2}\tilde{v} \right) \geq \delta \left( \tilde{h}^{'} \right)^{k-l}.
\end{equation*}
Since $\tilde{h}^{'}\geq \alpha_{0}(1+|D\tilde{v}|)$ if
$\tilde{h}^{'}$ is chosen sufficiently large ($\alpha_{0}$ positive
constant), we can get \eqref{f9} by taking $\mathcal{B}_{3}$ and
$\mathcal{B}_{4}$ sufficiently large such that $\tilde{h}^{'}$ are
large. In particular, we thus have: $\forall x\in M^{n}\cap
B_{r}(x_{0})$, $D^{2}\tilde{v}(x)\in \Gamma_{k}\left( Du(x)
\right)$. We may then use the concavity of $\left(
\frac{\mathcal{F}_{k}}{\mathcal{F}_{l}} \right)^{\frac{1}{k-l}}$
with respect to $q$ to bound $\sum_{i,j}\frac{\partial}{\partial
q_{ij}}\left( \frac{\mathcal{F}_{k}}{\mathcal{F}_{l}}(u)
\right)\tilde{v}_{ij}$ from below:

\begin{equation*}
\frac{1}{k-l}\sum_{i,j}\frac{\partial}{\partial q_{ij}}\left(  \frac{\mathcal{F}_{k}}{\mathcal{F}_{l}}(u) \right)\tilde{v}_{ij} \geq \left( \frac{\mathcal{F}_{k}}{\mathcal{F}_{l}} \right)^{1-\frac{1}{k-l}}(u)\left( \frac{\mathcal{F}_{k}}{\mathcal{F}_{l}} \right)^{\frac{1}{k-l}}\left( Du, D^{2}\tilde{v} \right),
\end{equation*}
where $\left( \frac{\mathcal{F}_{k}}{\mathcal{F}_{l}}
\right)^{1-\frac{1}{k-l}}(u)$ is itself estimated from below. In
view of \eqref{f9}, $\left( \frac{\mathcal{F}_{k}}{\mathcal{F}_{l}}
\right)^{\frac{1}{k-l}}\left( Du, D^{2}\tilde{v} \right)$ is large
than $\mathcal{B}_{5}\left( 1+|D\tilde{v}| \right)$ if
$\mathcal{B}_{3}$ and $\mathcal{B}_{4}$ are suitable parameters,
where $\mathcal{B}_{5}$ is as large as desired, i.e. Lemma
\ref{lemma 6.3} holds for sufficiently large $\mathcal{B}_{3}$,
$\mathcal{B}_{4}$ under control.
\end{proof}

$\mathit{Estimate~of~mixed~second~derivatives}$. Let
$\{e_{1},\cdots,e_{n}\}$ still denote an adapted basis at the
boundary-point $x_{0}$ and let $t\in\{1,\cdots,n-1\}$. Our purpose
is to estimate $u_{tn}(x_{0})$.  For any $ x\in
\overline{M^{n}}\cap\overline{B}_{r}(x_{0})$, let $\xi = e_{t} +
\tilde{\rho}_{t}(x^{'})e_{n}$, where $\tilde{\rho}$ denotes the
function locally defined on $T_{x_{0}}\partial M^{n}$ whose graph is
$\partial M^{n}$. Set
$$\mathfrak{g}(x,p) = \left\langle p, \xi(x)\right\rangle = p_{t} + \tilde{\rho}_{t}(x^{'})p_{n}.$$

Using Lemma \ref{lemma 6.1} and Lemma \ref{lemma 6.2}, there exists
the $\mathcal{B} = \mathcal{B}\left( n,k,l,\rho,\mathcal{B}_{0},
\Vert\mathfrak{g}\Vert_{1,D_{r,\rho}} \right)$, $\mathcal{B}_{1} =
\mathcal{B}_{1}\left(
n,k,l,M^{n},\psi,\rho,\mathcal{B}_{0},\Vert\mathfrak{g}\Vert_{2,D_{r,\rho}}
\right)$ and $b = b\left( n,\mathcal{B}_{1},
\Vert\mathfrak{g}\Vert_{0,D_{r,\rho}}, \mathcal{B} \right)$
sufficiently large such that the function
\begin{equation*}
\begin{aligned}
\widetilde{\mathbb{W}} &=
\exp\left( -\mathcal{B}_{1}u_{\xi}(x_{0}) \right) \\
&\quad
-\exp\left( -\mathcal{B}_{1}u_{\xi} + \frac{\mathcal{B}\mathcal{B}_{1}}{2}\sum_{s=1}^{n-1}\left( u_{s} - u_{s}(x_{0}) \right)^{2} \right) - b|x-x_{0}|^{2}
\end{aligned}
\end{equation*}
satisfies
\begin{equation*}
\sum_{i,j}\frac{\partial}{\partial q_{ij}}\left( \frac{\mathcal{F}_{k}}{\mathcal{F}_{l}}(u) \right)\widetilde{\mathbb{W}}_{ij} \leq \mathcal{B}_{2}\left( 1+|D\widetilde{\mathbb{W}}| \right),
\end{equation*}
where $\mathcal{B}_{2} = \mathcal{B}_{2}\left( \mathcal{B}_{1},
b,r,\Vert\mathfrak{g}\Vert_{0,D_{r,\rho}}, \mathcal{B} \right)$.
Define
\begin{equation*}
\begin{aligned}
\tilde{\psi}(x^{'}) &=
\exp\left( -\mathcal{B}_{1}\varphi_{\xi}(x_{0}) \right) - \exp\left( -\mathcal{B}_{1}\varphi_{\xi}( x^{'} ) \right)\cdot\\
&\quad \exp\left(\mathcal{B}\mathcal{B}_{1}\sum_{s=1}^{n-1}\left(
\varphi_{s}( x^{'} ) - \varphi_{s}(x_{0}) \right)^{2}
+2\mathcal{B}_{1}\mathcal{B} \left( |\varphi_{n}( x^{'} )|^{2} +1
\right)|D\tilde{\rho}(x^{'})| ^{2}\right)
\end{aligned}
\end{equation*}
and $\tilde{v} = -a_{0}|x-x_{0}|^{2} - \tilde{h}(d) + \tilde{\psi}(
x^{'} )$, where $\tilde{h}(d) = \mathcal{B}_{3}\left(
1-e^{-\mathcal{B}_{4}d} \right)$. For suitable constants $a_{0}$,
$\mathcal{B}_{3}$, $\mathcal{B}_{4}$, from Lemma \ref{lemma 6.3}, we
have $\tilde{v} \leq \widetilde{\mathbb{W}}$ on $\partial\left(
M^{n}\cap B_{r}(x_{0}) \right)$ and
$\sum_{i,j}\frac{\partial}{\partial q_{ij}}\left(
\frac{\mathcal{F}_{k}}{\mathcal{F}_{l}}(u) \right)\tilde{v}_{ij}
\geq \mathcal{B}_{2}\left( 1+|D\tilde{v}| \right)$ on
$\overline{M^{n}}\cap \bar{B}_{r}(x_{0})$. By the comparison
principle in \cite[Lemma 4.5]{Bayard}, we have
$\tilde{v}\leq\widetilde{\mathbb{W}}$ on
$\overline{M^{n}}\cap\bar{B}_{r}(x_{0})$, and then since
$\tilde{v}(x_{0}) = \widetilde{\mathbb{W}}(x_{0})$, we get
$\tilde{v}_{n}(x_{0})\leq\widetilde{\mathbb{W}}_{n}(x_{0})$, i.e.,
$\tilde{v}_{n}(x_{0})\leq \mathcal{B}_{1}u_{tn}(x_{0})\exp\left(
-\mathcal{B}_{1}\varphi_{t}(x_{0}) \right)$. In other words,
$$u_{tn}(x_{0}) \geq \mathcal{B}_{6},$$
where $\mathcal{B}_{6} = \mathcal{B}_{6}\left(
n,k,l,M^{n},\psi,\rho,\Vert\varphi\Vert_{3,\overline{M^{n}}}
\right)$. To estimate $u_{tn}(x_{0})$ from above, we do the similar
argument with $\mathfrak{g}(x,p) = -p_{t} - \tilde{\rho}_{t}( x^{'}
)p_{n}$.

\vspace{5mm}

$\mathit{Estimate~of~normal~second~derivatives}$. Now, we want to
estimate an upper bound on $h_{nn}$. A lower bound on $h_{nn}$
easily follows from $\sigma_{1}(u)>0$ and the estimates of
tangential and mixed second derivatives. An upper bound on $h_{nn}$
on $\partial M^{n}$ amounts to a lower bound on
\begin{equation*}
\mathcal{A}_{k,l} = \frac{\partial}{\partial h_{nn}}\left( \frac{\mathcal{F}_{k}}{\mathcal{F}_{l}}(u) \right)
\end{equation*}
on $\partial M^{n}$ by a positive quantity under control, since
\begin{equation*}
\frac{\mathcal{F}_{k}}{\mathcal{F}_{l}}(u) = \mathcal{A}_{k,l}\cdot h_{nn} + \mathcal{B}_{k,l} = \psi,
\end{equation*}
where $\mathcal{B}_{k,l}$ depends only on $u$, $Du$, the tangential
and mixied second derivatives of $u$, which are already estimated.
\begin{lemma}\label{lemma 6.5}
$\mathcal{A}_{k,l}$ is given by
\begin{equation*}
\begin{aligned}
\mathcal{A}_{k,l} &= \frac{1}{u^{2}}\frac{u^{2} - |\partial\varphi|^{2}}{u^{2}v^{2}}\cdot
\frac{\mathcal{F}_{k-1}(\partial\varphi, \partial^{2}\varphi+u_{\gamma}\partial\gamma)\cdot\mathcal{F}_{l} -
\mathcal{F}_{l-1}(\partial\varphi, \partial^{2}\varphi + u_{\gamma}\partial\gamma)\cdot\mathcal{F}_{k}}{\mathcal{F}_{l}^{2}}\\
&=
\frac{1}{u^{2}}\frac{u^{2} - |\partial\varphi|^{2}}{u^{2}v^{2}}\cdot\mathcal{F}_{k-1,l-1}(\partial\varphi,\partial^{2}\varphi+u_{\gamma}\partial\gamma),
\end{aligned}
\end{equation*}
where $\partial$ denotes the tangential gradient and $\gamma$ the
future-directed unit normal to $\partial M^{n}$,
$\mathcal{F}_{k-1,l-1}:=\frac{\mathcal{F}_{k-1}\mathcal{F}_{l} -
\mathcal{F}_{l-1}\mathcal{F}_{k}}{\mathcal{F}_{l}^{2}} $.
\end{lemma}

We continue to proceed as follows: estimate $h_{nn}$ via the
previous method at a point $y$ of $\partial M^{n}$ where
$\mathcal{F}_{k-1,l-1}(\partial\varphi,\partial^{2}\varphi+u_{\gamma}\partial\gamma)$
is minimum. It implies a lower bound on
$\mathcal{F}_{k-1,l-1}(\partial\varphi,\partial^{2}\varphi+u_{\gamma}\partial\gamma)$
because
\begin{equation*}
\mathcal{F}_{k-1,l-1}(\partial\varphi,\partial^{2}\varphi + u_{\gamma}\partial\gamma)(y) \geq
\frac{\mathcal{F}_{k-1}(\partial\varphi,\partial^{2}\varphi + u_{\gamma}\partial\gamma)}{\mathcal{F}_{l}}(y)
\geq
\mathcal{B}_{6}\left( \frac{\mathcal{F}_{k}}{\mathcal{F}_{l}}(u) \right)(y).
\end{equation*}
For the last inequality see \cite{nmi}, and the positive constant
$\mathcal{B}_{6}$ depends on $n$, $k$, $l$, $\rho$, $M^{n}$. From
the definition of $y$, the function
$\mathcal{F}_{k-1,l-1}(\partial\varphi,\partial^{2}\varphi+u_{\gamma}\partial\gamma)$
admits itself a lower bound on $\partial M^{n}$, and so does
$\mathcal{A}_{k,l}$, which yields an estimate on the second normal
derivatives at every boundary-point.

Let us take $\mathfrak{g}(x,p) =
\mathcal{F}_{k-1,l-1}\left(\partial\varphi(
x^{'}),\partial^{2}\varphi( x^{'})+\langle p, \gamma( x^{'}
)\rangle\partial\gamma( x^{'})\right)$, where $x=(x^{'},x_{n})$ in
the basis $\{e_{1},\cdots,e_{n}\}$. A priori $\mathfrak{g}$ is
concave with respect to $p$ only for $k=2$, $l=0$, then
$\mathfrak{g}(x,p) = \mathcal{F}_{1}\left(\partial\varphi(
x^{'}),\partial^{2}\varphi( x^{'})+\langle p, \gamma( x^{'}
)\rangle\partial\gamma( x^{'})\right)$. The rest is almost the same
as the argument in \cite[pp. 27-28]{Bayard}, so we omit it.

\begin{remark}
\rm{ Nearly the whole part of our proof in this section is valid for
any $k=2,\cdots,n$, $0\leq l\leq k-2$. However, the auxiliary
function $\mathfrak{g}(x,p)$ in Lemma \ref{lemma 6.1} is concave
with respect to $p$, and then we need to use the constraint $k=2$,
$l=0$ to estimate the double normal second derivatives on the
boundary. }
\end{remark}

By the method of continuity and the a prior estimates obtained here,
using a similar argument to that in \cite[Section 5]{GLM}, the
existence and uniqueness of the PCP \ref{main equation} with $k=2$,
$l=0$ can be proven. This completes the proof of Theorem \ref{main
1.3}.

\section*{Acknowledgments}
This work is partially supported by the NSF of China (Grant Nos.
11801496 and 11926352), the Fok Ying-Tung Education Foundation
(China) and  Hubei Key Laboratory of Applied Mathematics (Hubei
University).

\end{document}